\documentclass{amsart}%
\usepackage{amssymb}
\usepackage{amsfonts}%
\usepackage{amsmath}%
\usepackage{graphicx}
\providecommand{\U}[1]{\protect \rule{.1in}{.1in}}
\newtheorem{theorem}{Theorem}
\theoremstyle{plain}

\newtheorem{conclusion}{Conclusion}

\newtheorem{corollary}{Corollary}

\newtheorem{definition}{Definition}

\newtheorem{lemma}{Lemma}

\newtheorem{proposition}{Proposition}
\newtheorem{remark}{Remark}

\numberwithin{equation}{section}
\begin{document}
\title[Generalized weighted Morrey estimates]{The boundedness of certain sublinear operators with rough kernel generated by
{Calder\'{o}n-Zygmund operators} and their commutators on generalized weighted
Morrey spaces}
\author{F.GURBUZ}
\address{ANKARA UNIVERSITY, FACULTY OF SCIENCE, DEPARTMENT OF MATHEMATICS, TANDO\u{G}AN
06100, ANKARA, TURKEY }
\curraddr{}
\email{feritgurbuz84@hotmail.com}
\urladdr{}
\thanks{}
\thanks{}
\thanks{}
\date{}
\subjclass[2010]{ 42B20, 42B25, 42B35}
\keywords{{Sublinear operator; Calder\'{o}n-Zygmund operator; rough kernel; generalized
weighted Morrey space; commutator; BMO}}
\dedicatory{ }
\begin{abstract}
The aim of this paper is to get the boundedness of certain sublinear operators
with rough kernel generated by {Calder\'{o}n-Zygmund operators} on the
generalized weighted Morrey spaces under generic size conditions which are
satisfied by most of the operators in harmonic analysis. We also prove that
the commutator operators formed by $BMO$ functions and certain sublinear
operators with rough kernel are also bounded on the generalized weighted
Morrey spaces. Marcinkiewicz operator which satisfies the conditions of these
theorems can be considered as an example.

\end{abstract}
\maketitle

\section{Introduction}

The classical Morrey spaces $M_{p,\lambda}$ have been introduced by Morrey in
\cite{Morrey} to study the local behavior of solutions of second order
elliptic partial differential equations(PDEs). In recent years there has been
an explosion of interest in the study of the boundedness of operators on
Morrey-type spaces. It has been obtained that many properties of solutions to
PDEs are concerned with the boundedness of some operators on Morrey-type
spaces. In fact, better inclusion between Morrey and H\"{o}lder spaces allows
to obtain higher regularity of the solutions to different elliptic and
parabolic boundary problems (see \cite{FazPalRag, Pal, Shi, Softova} for details).

Let ${\mathbb{R}^{n}}$ be the $n-$dimensional Euclidean space of points
$x=(x_{1},...,x_{n})$ with norm $|x|=\left(
%TCIMACRO{\dsum \limits_{i=1}^{n}}%
%BeginExpansion
{\displaystyle \sum \limits_{i=1}^{n}}
%EndExpansion
x_{i}^{2}\right)  ^{\frac{1}{2}}$. Let $B=B(x_{0},r_{B}) $ denote the ball
with the center $x_{0}$ and radius $r_{B}$. For a given measurable set $E$, we
also denote the Lebesgue measure of $E$ by $\left \vert E\right \vert $. For any
given $\Omega \subseteq{\mathbb{R}^{n}}$ and $0<p<\infty$, denote by
$L_{p}\left(  \Omega \right)  $ the spaces of all functions $f$ satisfying%
\[
\left \Vert f\right \Vert _{L_{p}\left(  \Omega \right)  }=\left(
%TCIMACRO{\dint \limits_{\Omega}}%
%BeginExpansion
{\displaystyle \int \limits_{\Omega}}
%EndExpansion
\left \vert f\left(  x\right)  \right \vert ^{p}dx\right)  ^{\frac{1}{p}}%
<\infty.
\]

We recall the definition of classical Morrey spaces $M_{p,\lambda}$ as%

\[
M_{p,\lambda}\left(  {\mathbb{R}^{n}}\right)  =\left \{  f:\left \Vert
f\right \Vert _{M_{p,\lambda}\left(  {\mathbb{R}^{n}}\right)  }=\sup
\limits_{x\in{\mathbb{R}^{n}},r>0}\,r^{-\frac{\lambda}{p}}\, \Vert
f\Vert_{L_{p}(B(x,r))}<\infty \right \}  ,
\]
where $f\in L_{p}^{loc}({\mathbb{R}^{n}})$, $0\leq \lambda \leq n$ and $1\leq
p<\infty$.

Note that $M_{p,0}=L_{p}({\mathbb{R}^{n}})$ and $M_{p,n}=L_{\infty
}({\mathbb{R}^{n}})$. If $\lambda<0$ or $\lambda>n$, then $M_{p,\lambda
}={\Theta}$, where $\Theta$ is the set of all functions equivalent to $0$ on
${\mathbb{R}^{n}}$. It is known that $M_{p,\lambda}({\mathbb{R}^{n}})$ is an
expansion of $L_{p}({\mathbb{R}^{n}})$ in the sense that $L_{p,0}%
=L_{p}({\mathbb{R}^{n}})$.

We also denote by $WM_{p,\lambda}\equiv WM_{p,\lambda}({\mathbb{R}^{n}})$ the
weak Morrey space of all functions $f\in WL_{p}^{loc}({\mathbb{R}^{n}})$ for
which
\[
\left \Vert f\right \Vert _{WM_{p,\lambda}}\equiv \left \Vert f\right \Vert
_{WM_{p,\lambda}({\mathbb{R}^{n}})}=\sup_{x\in{\mathbb{R}^{n}},\;r>0}%
r^{-\frac{\lambda}{p}}\Vert f\Vert_{WL_{p}(B(x,r))}<\infty,
\]
where $WL_{p}(B(x,r))$ denotes the weak $L_{p}$-space of measurable functions
$f$ for which
\[%
\begin{split}
\Vert f\Vert_{WL_{p}(B(x,r))} &  \equiv \Vert f\chi_{_{B(x,r)}}\Vert
_{WL_{p}({\mathbb{R}^{n}})}\\
&  =\sup_{t>0}t\left \vert \left \{  y\in B(x,r):\,|f(y)|>t\right \}  \right \vert
^{1/{p}}\\
&  =\sup_{0<t\leq|B(x,r)|}t^{1/{p}}\left(  f\chi_{_{B(x,r)}}\right)  ^{\ast
}(t)<\infty,
\end{split}
\]
where $g^{\ast}$ denotes the non-increasing rearrangement of a function $g$.

Throughout the paper we assume that $x\in{\mathbb{R}^{n}}$ and $r>0$ and also
let $B(x,r)$ denotes the open ball centered at $x$ of radius $r$, $B^{C}(x,r)$
denotes its complement and $|B(x,r)|$ is the Lebesgue measure of the ball
$B(x,r)$ and $|B(x,r)|=v_{n}r^{n}$, where $v_{n}=|B(0,1)|$.

Morrey has investigated that many properties of solutions to PDEs can be
attributed to the boundedness of some operators on Morrey spaces. For the
boundedness of the Hardy--Littlewood maximal operator, the fractional integral
operator and the Calder\'{o}n--Zygmund singular integral operator on these
spaces, we refer the readers to \cite{Adams, ChFra, Peetre}. For the
properties and applications of classical Morrey spaces, see \cite{ChFraL1,
ChFraL2, FazRag2, FazPalRag} and references therein.

After studying Morrey spaces in detail, researchers have passed to generalized
Morrey spaces. Mizuhara \cite{Miz} has given generalized Morrey spaces
$M_{p,\varphi}$ considering $\varphi \left(  r\right)  $ instead of
$r^{\lambda}$ in the above definition of the Morrey space. Later, Guliyev
\cite{GulJIA} and Karaman \cite{Karaman} have defined the generalized Morrey
spaces $M_{p,\varphi}$ with normalized norm as follows:

\begin{definition}
$\left(  \text{\textbf{Generalized Morrey space}}\right)  $ Let $\varphi(x,r)$
be a positive measurable function on ${\mathbb{R}^{n}}\times(0,\infty)$ and
$1\leq p<\infty$. We denote by $M_{p,\varphi}\equiv M_{p,\varphi}%
({\mathbb{R}^{n}})$ the generalized Morrey space, the space of all functions
$f\in L_{p}^{loc}({\mathbb{R}^{n}})$ with finite quasinorm
\[
\Vert f\Vert_{M_{p,\varphi}}=\sup \limits_{x\in{\mathbb{R}^{n}},r>0}%
\varphi(x,r)^{-1}\,|B(x,r)|^{-\frac{1}{p}}\, \Vert f\Vert_{L_{p}(B(x,r))}.
\]
Also by $WM_{p,\varphi}\equiv WM_{p,\varphi}({\mathbb{R}^{n}})$ we denote the
weak generalized Morrey space of all functions $f\in WL_{p}^{loc}%
({\mathbb{R}^{n}})$ for which
\[
\Vert f\Vert_{WM_{p,\varphi}}=\sup \limits_{x\in{\mathbb{R}^{n}},r>0}%
\varphi(x,r)^{-1}\,|B(x,r)|^{-\frac{1}{p}}\, \Vert f\Vert_{WL_{p}%
(B(x,r))}<\infty.
\]

\end{definition}

According to this definition, we recover the Morrey space $M_{p,\lambda}$ and
weak Morrey space $WM_{p,\lambda}$ under the choice $\varphi(x,r)=r^{\frac
{\lambda-n}{p}}$:
\[
M_{p,\lambda}=M_{p,\varphi}\mid_{\varphi(x,r)=r^{\frac{\lambda-n}{p}}%
},~~~~~~~~WM_{p,\lambda}=WM_{p,\varphi}\mid_{\varphi(x,r)=r^{\frac{\lambda
-n}{p}}}.
\]

During the last decades various classical operators, such as maximal, singular
and potential operators have been widely investigated in classical and
generalized Morrey spaces (see \cite{BGGS, Gurbuz, Karaman, lu-yang-zhou} for details).

Maximal functions and singular integrals play a key role in harmonic analysis
since maximal functions could control crucial quantitative information
concerning the given functions, despite their larger size, while singular
integrals, Hilbert transform as it's prototype, recently intimately connected
with PDEs, operator theory and other fields.

Let $f\in L^{loc}\left(  {\mathbb{R}^{n}}\right)  $. The
Hardy-Littlewood(H--L) maximal operator $M$ is defined by
\[
Mf(x)=\sup_{t>0}|B(x,t)|^{-1}\int \limits_{B(x,t)}|f(y)|dy.
\]

Let $\overline{T}$ be a standard Calder\'{o}n-Zygmund(C--Z) singular integral
operator, briefly a C--Z operator, i.e., a linear operator bounded from
$L_{2}({\mathbb{R}^{n}})$ to $L_{2}({\mathbb{R}^{n}})$ taking all infinitely
continuously differentiable functions $f$ with compact support to the
functions $f\in L_{1}^{loc}({\mathbb{R}^{n}})$ represented by
\[
\overline{T}f(x)=p.v.\int \limits_{{\mathbb{R}^{n}}}k(x-y)f(y)\,dy\qquad
x\notin suppf.
\]
Such operators have been introduced in \cite{CM}. Here $k$ is a C--Z kernel
\cite{Grafakos}. Chiarenza and Frasca \cite{ChFra} have obtained the
boundedness of H--L maximal operator $M$ and C--Z operator $\overline{T}$ on
$M_{p,\lambda}\left(  {\mathbb{R}^{n}}\right)  $. It is also well known that
H--L maximal operator $M$ and C--Z operator $\overline{T}$ play an important
role in harmonic analysis (see \cite{GarRub, LuDingY, St, Stein93, Torch}).
Also, the theory of the C--Z operator is one of the important achievements of
classical analysis in the last century, which has many important applications
in Fourier analysis, complex analysis, operator theory and so on.

Suppose that $S^{n-1}$ is the unit sphere in ${\mathbb{R}^{n}}$ $(n\geq2)$
equipped with the normalized Lebesgue measure $d\sigma$. Let $\Omega \in
L_{s}(S^{n-1})$ with $1<s\leq \infty$ be homogeneous of degree zero. We define
$s^{\prime}=\frac{s}{s-1}$ for any $s>1$. Suppose that $T_{\Omega}$ represents
a linear or a sublinear operator, which satisfies that for any $f\in
L_{1}({\mathbb{R}^{n}})$ with compact support and $x\notin suppf$
\begin{equation}
|T_{\Omega}f(x)|\leq c_{0}\int \limits_{{\mathbb{R}^{n}}}\frac{|\Omega
(x-y)|}{|x-y|^{n}}\,|f(y)|\,dy,\label{e1}%
\end{equation}
where $c_{0}$ is independent of $f$ and $x$.

For a locally integrable function $b$ on ${\mathbb{R}^{n}}$, suppose that the
commutator operator $T_{\Omega,b}$ represents a linear or a sublinear
operator, which satisfies that for any $f\in L_{1}({\mathbb{R}^{n}})$ with
compact support and $x\notin suppf$
\begin{equation}
|T_{\Omega,b}f(x)|\leq c_{0}%
%TCIMACRO{\dint \limits_{{\mathbb{R}^{n}}}}%
%BeginExpansion
{\displaystyle \int \limits_{{\mathbb{R}^{n}}}}
%EndExpansion
|b(x)-b(y)|\, \frac{|\Omega(x-y)|}{|x-y|^{n}}\,|f(y)|\,dy,\label{e2}%
\end{equation}
where $c_{0}$ is independent of $f$ and $x$.

We point out that the condition (\ref{e1}) in the case $\Omega \equiv1$ was
first introduced by Soria and Weiss in \cite{SW} . The conditions (\ref{e1})
and (\ref{e2}) are satisfied by many interesting operators in harmonic
analysis, such as the C--Z operators, Carleson's maximal operator, H--L
maximal operator, C. Fefferman's singular multipliers, R. Fefferman's singular
integrals, Ricci--Stein's oscillatory singular integrals, the Bochner--Riesz
means and so on (see \cite{LLY}, \cite{SW} for details).

Let $\Omega \in L_{s}(S^{n-1})$ with $1<s\leq \infty$ be homogeneous of degree
zero and satisfies the cancellation condition
\[
\int \limits_{S^{n-1}}\Omega(x^{\prime})d\sigma(x^{\prime})=0,
\]
where $x^{\prime}=\frac{x}{|x|}$ for any $x\neq0$. The C--Z singular integral
operator with rough kernel $\overline{T}_{\Omega}$ is defined by%

\[
\overline{T}_{\Omega}f(x)=p.v.\int \limits_{{\mathbb{R}^{n}}}\frac{\Omega
(x-y)}{|x-y|^{n}}f(y)dy,
\]
satisfies the condition (\ref{e1}).

It is obvious that when $\Omega \equiv1$, $\overline{T}_{\Omega}$ is the C--Z
operator $\overline{T}$.

In 1976, Coifman et al. \cite{CRW} introduced the commutator generated by
$\overline{T}_{\Omega}$ and a local integrable function $b$ as follows:
\begin{equation}
\lbrack b,\overline{T}_{\Omega}]f(x)\equiv b(x)\overline{T}_{\Omega
}f(x)-\overline{T}_{\Omega}(bf)(x)=p.v.\int \limits_{{\mathbb{R}^{n}}%
}[b(x)-b(y)]\frac{\Omega(x-y)}{|x-y|^{n}}f(y)dy.\label{e3}%
\end{equation}
Sometimes, the commutator defined by (\ref{e3}) is also called the commutator
in Coifman-Rocherberg-Weiss's sense, which has its root in the complex
analysis and harmonic analysis (see \cite{CRW}).

\begin{remark}
\cite{Shi} As another extension of Hilbert transform, a variety of operators
related to the singular integrals for C--Z with homogeneous kernel, but
lacking the smoothness required in the classical theory, have been studied. In
this case, when $\Omega$ satisfies some size conditions, the kernel of the
operator has no regularity, and so the operator is called rough integral
operator. The theory of Operators with homogeneous kernel is a well studied
area (see \cite{Grafakos} and \cite{LuDingY} for example). Lu et al.
(\cite{lu-yang-zhou}), Gurbuz et al. (\cite{BGGS}) and Gurbuz (\cite{Gurbuz})
have studied certain sublinear operators mentioned above with rough kernel on
the generalized Morrey spaces. These include the operator $[b,\overline
{T}_{\Omega}]$. For more results, we refer the reader to \cite{BGGS, Gurbuz,
Gurbuz1, lu-yang-zhou, LuDingY}.
\end{remark}

In \cite{BGGS, Gurbuz}, the boundedness of the sublinear operators with rough
kernel generated by C--Z operators and their commutators on generalized Morrey
spaces has been investigated.

In this paper, we first prove the boundedness of the sublinear operators with
rough kernels $T_{\Omega}$ satisfying condition (\ref{e1}) generated by C--Z
singular integral operators with rough kernel from one generalized weighted
Morrey space $M_{p,\varphi_{1}}\left(  w\right)  $ to another $M_{p,\varphi
_{2}}\left(  w\right)  $ with the weight function $w$ belonging to
Muckenhoupt's class $A_{p}$ for $1<p<\infty$, and from the space
$M_{1,\varphi_{1}}\left(  w\right)  $ to the weak space $WM_{1,\varphi_{2}%
}\left(  w\right)  $. Then, we also obtain the boundedness of the sublinear
commutator operators $T_{\Omega,b}$ satisfying condition (\ref{e2}) generated
by a C--Z type operator with rough kernel and $b$ from one generalized
weighted Morrey space $M_{p,\varphi_{1}}\left(  w\right)  $ to another
$M_{p,\varphi_{2}}\left(  w\right)  $ for $1<p<\infty$, $b\in BMO$ (bounded
mean oscillation). Provided that $b\in BMO$ and $T_{\Omega,b}$ is a sublinear
operator, we find the sufficient conditions on the pair $(\varphi_{1}%
,\varphi_{2})$ which ensures the boundedness of the commutator operators
$T_{\Omega,b}$ from $M_{p,\varphi_{1}}\left(  w\right)  $ to another
$M_{p,\varphi_{2}}\left(  w\right)  $ for $1<p<\infty$. In all the cases the
conditions for the boundedness of $T_{\Omega}$ and $T_{\Omega,b}$ are given in
terms of Zygmund-type integral inequalities on $\left(  \varphi_{1}%
,\varphi_{2}\right)  $ which do not assume any assumption on monotonicity of
$\varphi_{1},\varphi_{2}$ in $r$. Finally, as an example to the conditions of
these theorems are satisfied, we can consider the Marcinkiewicz operator.

By $A\lesssim B$ we mean that $A\leq CB$ with some positive constant $C$
independent of appropriate quantities. If $A\lesssim B$ and $B\lesssim A$, we
write $A\approx B$ and say that $A$ and $B$ are equivalent. We will also
denote the conjugate exponent of $p>1$ by $p^{\prime}=\frac{p}{p-1}$ and $s>1$
by $s^{\prime}=\frac{s}{s-1}$.

\section{Weighted Morrey spaces}

A weight function is a locally integrable function on ${\mathbb{R}^{n}}$ which
takes values in $(0,\infty)$ almost everywhere. For a weight function $w$ and
a measurable set $E$, we define $w(E)=%
%TCIMACRO{\dint \limits_{E}}%
%BeginExpansion
{\displaystyle \int \limits_{E}}
%EndExpansion
w(x)dx$, the Lebesgue measure of $E$ by $|E|$ and the characteristic function
of $E$ by $\chi_{_{E}}$. Given a weight function $w$, we say that $w$
satisfies the doubling condition if there exists a constant $D>0$ such that
for any ball $B $, we have $w(2B)\leq Dw(B)$. When $w$ satisfies this
condition, we denote $w\in \Delta_{2}$, for short.

If $w$ is a weight function, we denote by $L_{p}(w)\equiv L_{p}({{\mathbb{R}%
^{n}}},w)$ the weighted Lebesgue space defined by the norm
\[
\Vert f\Vert_{L_{p,w}}=\left(
%TCIMACRO{\dint \limits_{{{\mathbb{R}^{n}}}}}%
%BeginExpansion
{\displaystyle \int \limits_{{{\mathbb{R}^{n}}}}}
%EndExpansion
|f(x)|^{p}w(x)dx\right)  ^{\frac{1}{p}}<\infty,\qquad \text{when }1\leq
p<\infty
\]
and by $\Vert f\Vert_{L_{\infty,w}}=\operatorname*{esssup}\limits_{x\in
{\mathbb{R}^{n}}}|f(x)|w(x)$ when $p=\infty$.

We denote by $WL_{p}(w)$ the weighted weak space consisting of all measurable
functions $f$ such that%
\[
\Vert f\Vert_{WL_{p}(w)}=\sup \limits_{t>0}tw\left(  \left \{  x\in
{{\mathbb{R}^{n}:}}\left \vert f\left(  x\right)  \right \vert >t\right \}
\right)  ^{\frac{1}{p}}<\infty.
\]

We recall that a weight function $w$ is in the Muckenhoupt's class
$A_{p}\left(  {{\mathbb{R}^{n}}}\right)  $, $1<p<\infty$, if
\begin{align}
\lbrack w]_{A_{p}} &  :=\sup \limits_{B}[w]_{A_{p}(B)}\nonumber \\
&  =\sup \limits_{B}\left(  \frac{1}{|B|}%
%TCIMACRO{\dint \limits_{B}}%
%BeginExpansion
{\displaystyle \int \limits_{B}}
%EndExpansion
w(x)dx\right)  \left(  \frac{1}{|B|}%
%TCIMACRO{\dint \limits_{B}}%
%BeginExpansion
{\displaystyle \int \limits_{B}}
%EndExpansion
w(x)^{1-p^{\prime}}dx\right)  ^{p-1}<\infty,\label{2}%
\end{align}
where the supremum is taken with respect to all the balls $B$ and $\frac{1}%
{p}+\frac{1}{p^{\prime}}=1$. The expression $[w]_{A_{p}}$ is called
characteristic constant of $w$. Note that, for all balls $B$ we have
\begin{equation}
\lbrack w]_{A_{p}}^{1/p}\geq \lbrack w]_{A_{p}(B)}^{1/p}=|B|^{-1}\Vert
w\Vert_{L_{1}(B)}^{1/p}\Vert w^{-1/p}\Vert_{L_{p^{\prime}}(B)}\geq1\label{1}%
\end{equation}
by the H\"{o}lder's inequality. For $p=1$, the class $A_{1}\left(
{{\mathbb{R}^{n}}}\right)  $ is defined by
\begin{equation}
\frac{1}{|B|}%
%TCIMACRO{\dint \limits_{B}}%
%BeginExpansion
{\displaystyle \int \limits_{B}}
%EndExpansion
w(x)dx\leq C\inf \limits_{x\in B}w\left(  x\right) \label{5}%
\end{equation}
for every ball $B\subset{{\mathbb{R}^{n}}}$. Thus, we have the condition
$Mw(x)\leq Cw(x)$ with $[w]_{A_{1}}=\sup \limits_{x\in{\mathbb{R}^{n}}}%
\frac{Mw(x)}{w(x)}$, and also for $p=\infty$ we define $A_{\infty}=%
%TCIMACRO{\dbigcup \limits_{1\leq p<\infty}}%
%BeginExpansion
{\displaystyle \bigcup \limits_{1\leq p<\infty}}
%EndExpansion
A_{p}$, $[w]_{A_{\infty}}=\inf \limits_{1\leq p<\infty}[w]_{A_{p}}$ and
$[w]_{A_{\infty}}\leq \lbrack w]_{A_{p}}$.

One knows that $A_{p}\subset A_{s}$ if $1\leq p<s<\infty$, and that $w\in$
$A_{p}$ for some $1<p<s$ if $w\in A_{s}$ with $s>1$, and also $[w]_{A_{p}}%
\leq \lbrack w]_{A_{s}}$.

By (\ref{2}), we have
\begin{equation}
\left(  w^{-\frac{p^{\prime}}{p}}\left(  B\right)  \right)  ^{\frac
{1}{p^{\prime}}}=\left \Vert w^{-\frac{1}{p}}\right \Vert _{L_{p^{\prime}%
}\left(  B\right)  }\leq C\left \vert B\right \vert w\left(  B\right)
^{-\frac{1}{p}}\label{3}%
\end{equation}
for $1<p<\infty$. Note that%
\begin{equation}
\left(  \operatorname*{essinf}\limits_{x\in E}f\left(  x\right)  \right)
^{-1}=\operatorname*{esssup}\limits_{x\in E}\frac{1}{f\left(  x\right)
}\label{4}%
\end{equation}
is true for any real-valued nonnegative function $f$ and is measurable on $E$
(see \cite{Wheeden-Zygmund} page 143) and (\ref{5}); we get%
\begin{align}
\left \Vert w^{-1}\right \Vert _{L_{\infty}\left(  B\right)  }  &
=\operatorname*{esssup}\limits_{x\in B}\frac{1}{w\left(  x\right)
}\nonumber \\
& =\frac{1}{\operatorname*{essinf}\limits_{x\in B}w\left(  x\right)  }\leq
C\left \vert B\right \vert w\left(  B\right)  ^{-1}.\label{6*}%
\end{align}

\begin{proposition}
$\left(  \text{see \cite{Kuzu}}\right)  $ Since definition of the
Muckenhoupt's class $A_{p}\left(  {{\mathbb{R}^{n}}}\right)  $, we have%
\[
w^{1-p^{\prime}}\in A_{\frac{p^{\prime}}{s^{\prime}}}\text{ implies
}[w^{1-p^{\prime}}]_{A_{_{\frac{p^{\prime}}{s^{\prime}}}}\left(  B\right)
}^{\frac{s^{\prime}}{p^{\prime}}}=|B|^{-1}\Vert w^{^{1-p^{\prime}}}%
\Vert_{L_{1}(B)}^{\frac{s^{\prime}}{p^{\prime}}}\Vert w^{\frac{s^{\prime}}{p}%
}\Vert_{L_{\left(  \frac{p^{\prime}}{s^{\prime}}\right)  ^{\prime}}(B)}%
\]
for $1<p<\infty$. Since $w^{1-p^{\prime}}\in A_{\frac{p^{\prime}}{s^{\prime}}%
}\subset A_{p^{\prime}}$, we also know $w^{1-p^{\prime}}\in A_{\frac
{p^{\prime}}{s^{\prime}}}$ implies $w^{1-p^{\prime}}\in A_{p^{\prime}}$. Thus,
we have%
\begin{equation}
\lbrack w^{1-p^{\prime}}]_{A_{p^{\prime}}\left(  B\right)  }^{\frac
{1}{p^{\prime}}}=|B|^{-1}\Vert w^{^{1-p^{\prime}}}\Vert_{L_{1}(B)}^{\frac
{1}{p^{\prime}}}\Vert w^{\frac{1}{p}}\Vert_{L_{^{p}}(B)}.\label{7}%
\end{equation}
But, the converse of this implication is not generally valid.
\end{proposition}

\begin{proposition}
$\left(  \text{see \cite{Kuzu}}\right)  $ To make the proofs simpler, we can
also write $w^{1-p^{\prime}}\in A_{\frac{p^{\prime}}{s^{\prime}}}$ as follows:%
\begin{align}
\lbrack w^{1-p^{\prime}}]_{A_{\frac{p^{\prime}}{s^{\prime}}}\left(  B\right)
}^{\frac{s\left(  p-1\right)  }{p\left(  s-1\right)  }}  & =|B|^{-1}\Vert
w^{^{1-p^{\prime}}}\Vert_{L_{1}(B)}^{\frac{s\left(  p-1\right)  }{p\left(
s-1\right)  }}\Vert w^{\frac{s^{\prime}}{p}}\Vert_{L_{\left(  \frac{p^{\prime
}}{s^{\prime}}\right)  ^{\prime}}(B)}\nonumber \\
\lbrack w^{1-p^{\prime}}]_{A_{_{\frac{p^{\prime}}{s^{\prime}}}}\left(
B\right)  }^{\frac{1}{p^{\prime}}}  & =|B|^{-\frac{s-1}{s}}\Vert
w^{^{1-p^{\prime}}}\Vert_{L_{1}(B)}^{\frac{1}{p^{\prime}}}\Vert w\Vert
_{L_{^{\frac{s}{s-p}}}(B)}^{\frac{1}{p}},\label{8}%
\end{align}
where%
\[
1-p^{\prime}=-\frac{p^{\prime}}{p},~\frac{s^{\prime}}{p}=\frac{s}{p\left(
s-1\right)  },~\frac{s^{\prime}}{p^{\prime}}=\frac{s\left(  p-1\right)
}{p\left(  s-1\right)  },~\left(  \frac{s}{p}\right)  ^{\prime}=\frac{s}%
{s-p},~\left(  \frac{p^{\prime}}{s^{\prime}}\right)  ^{\prime}=\frac{p\left(
s-1\right)  }{s-p}.
\]

In the equation (\ref{8}) if we write (\ref{7}) instead of $\Vert
w^{^{1-p^{\prime}}}\Vert_{L_{1}(B)}^{\frac{1}{p^{\prime}}}$, then we obtain%
\begin{equation}
\lbrack w^{1-p^{\prime}}]_{A_{_{\frac{p^{\prime}}{s^{\prime}}}}\left(
B\right)  }^{\frac{1}{p^{\prime}}}=|B|^{\frac{1}{s}}[w^{1-p^{\prime}%
}]_{A_{p^{\prime}}\left(  B\right)  }^{\frac{1}{p^{\prime}}}\Vert w^{\frac
{1}{p}}\Vert_{L_{^{p}}(B)}^{-1}\Vert w\Vert_{L_{^{\frac{s}{s-p}}}(B)}%
^{\frac{1}{p}}.\label{9}%
\end{equation}

\end{proposition}

\begin{lemma}
\label{lemma10}$\left(  \text{see \cite{Kuzu}}\right)  $ Let $1<p<s$ and
$w^{1-p^{\prime}}\in A_{\frac{p^{\prime}}{s^{\prime}}}$. Then, the inequality%
\[
\left \Vert \Omega \left(  \cdot-y\right)  \right \Vert _{L_{p,w}\left(
B\right)  }\lesssim \left \Vert \Omega \left(  \cdot-y\right)  \right \Vert
_{L_{s}\left(  B\right)  }\left \Vert w\right \Vert _{L_{\left(  \frac{s}%
{p}\right)  ^{\prime}}\left(  B\right)  }^{\frac{1}{p}}%
\]
holds for every $y\in{{\mathbb{R}^{n}}}$ and for any ball $B\subset
{{\mathbb{R}^{n}}}$.
\end{lemma}

The classical $A_{p}\left(  {{\mathbb{R}^{n}}}\right)  $ weight theory has
been introduced by Muckenhoupt in the study of weighted $L_{p}$-boundedness of
H--L maximal function in \cite{Muckenhoupt}.

It is known from \cite{Grafakos} that

\begin{lemma}
\label{lemma100}The following statements hold.

$(1)~$ If $w\in A_{p}$ for some $1\leq p<\infty$, then $w\in \Delta_{2}$.
Moreover, for all $\lambda>1$ we have
\[
w(\lambda B)\leq \lambda^{np}[w]_{A_{p}}w(B).
\]

$(2)~$ If $w\in A_{\infty}$, then $w\in \Delta_{2}$. Moreover, for all
$\lambda>1$ we have
\[
w(\lambda B)\leq2^{\lambda^{n}}[w]_{A_{\infty}}^{\lambda^{n}}w(B).
\]

$(3)~$ If $w\in A_{p}$ for some $1\leq p\leq \infty$, then there exit $C>0$ and
$\delta>0$ such that for any ball $B$ and a measurable set $S\subset B$,
\[
\frac{1}{[w]_{A_{p}}}\left(  \frac{|S|}{|B|}\right)  \leq \frac{w(S)}{w(B)}\leq
C\left(  \frac{|S|}{|B|}\right)  ^{\delta}.
\]

$(4)~$The function $w^{-\frac{1}{p-1}}$ is in $A_{p^{\prime}}$ where $\frac
{1}{p}+\frac{1}{p^{\prime}}=1$, $1<p<\infty$ with characteristic constant%
\[
\lbrack w^{-\frac{1}{p-1}}]_{A_{p^{\prime}}}=[w]_{A_{p}}^{\frac{1}{p-1}}.
\]

\end{lemma}

Komori and Shirai \cite{KomShir} have introduced a version of the weighted
Morrey space $L_{p,\kappa}(w)$, which is a natural generalization of the
weighted Lebesgue space $L_{p}(w)$, and have investigated the boundedness of
classical operators in harmonic analysis.

\begin{definition}
$\left(  \text{\textbf{Weighted Morrey space}}\right)  $ Let $1\leq p<\infty$,
$0<\kappa<1$ and $w$ be a weight function. We denote by $L_{p,\kappa}(w)\equiv
L_{p,\kappa}({\mathbb{R}^{n}},w)$ the weighted Morrey space of all classes of
locally integrable functions $f$ with the norm
\[
\Vert f\Vert_{L_{p,\kappa}(w)}=\sup \limits_{x\in{\mathbb{R}^{n}}%
,r>0}\,w(B(x,r))^{-\frac{\kappa}{p}}\, \Vert f\Vert_{L_{p,w}(B(x,r))}<\infty.
\]

Furthermore, by $WL_{p,\kappa}(w)\equiv WL_{p,\kappa}({\mathbb{R}^{n}},w)$ we
denote the weak weighted Morrey space of all classes of locally integrable
functions $f$ with the norm
\[
\|f\|_{WL_{p,\kappa}(w)} = \sup \limits_{x\in{\mathbb{R}^{n}}, r>0} \,
w(B(x,r))^{-\frac{\kappa}{p}} \, \|f\|_{WL_{p,w}(B(x,r))} < \infty.
\]

\end{definition}

\begin{remark}
Alternatively, we could define the weighted Morrey spaces with cubes instead
of balls. Hence we shall use these two definitions of weighted Morrey spaces
appropriate to calculation.
\end{remark}

\begin{remark}
$(1)~$ If $w\equiv1$ and $\kappa=\lambda/n$ with $0\leq \lambda \leq n$, then
$L_{p,\lambda/n}(1)=M_{p,\lambda}({{\mathbb{R}^{n}}})$ is the classical Morrey spaces.

$(2)~$ If $\kappa=0,$ then $L_{p,0}(w)=L_{p}(w)$ is the weighted Lebesgue spaces.
\end{remark}

The following theorem has been proved in \cite{KomShir}.

\begin{theorem}
\label{teo9*}Let $1\leq p<\infty$, $0<\kappa<1$ and $w\in A_{p}$. Then the
operators $M$ and $\overline{T}$ are bounded on $L_{p,\kappa}\left(  w\right)
$ for $p>1$ and from $L_{1,\kappa}\left(  w\right)  $ to $WL_{1,\kappa}\left(
w\right)  $.
\end{theorem}

\section{Sublinear operators with rough kernel generated by
{Calder\'{o}n-Zygmund operators} on the generalized weighted Morrey spaces
$M_{p,\varphi}\left(  w\right)  $}

The generalized weighted Morrey spaces $M_{p,\varphi}\left(  w\right)  $ have
been introduced by Guliyev \cite{Guliyev} and Karaman \cite{Karaman} as follows.

\begin{definition}
$\left(  \text{\textbf{Generalized weighted Morrey space}}\right)  $ Let
$1\leq p<\infty$, $\varphi(x,r)$ be a positive measurable function on
${\mathbb{R}^{n}}\times(0,\infty)$ and $w$ be non-negative measurable function
on ${\mathbb{R}^{n}}$. We denote by $M_{p,\varphi}(w)\equiv M_{p,\varphi
}({\mathbb{R}^{n}},w)$ the generalized weighted Morrey space, the space of all
classes of functions $f\in L_{p,w}^{loc}({\mathbb{R}^{n}})$ with finite norm
\[
\Vert f\Vert_{M_{p,\varphi}(w)}=\sup \limits_{x\in{\mathbb{R}^{n}},r>0}%
\varphi(x,r)^{-1}\,w(B(x,r))^{-\frac{1}{p}}\, \Vert f\Vert_{L_{p,w}(B(x,r))},
\]
where $L_{p,w}(B(x,r))$ denotes the weighted $L_{p,w}$-space of measurable
functions $f$ for which
\[
\Vert f\Vert_{L_{p,w}\left(  B\left(  x,r\right)  \right)  }\equiv \Vert
f\chi_{B\left(  x,r\right)  }\Vert_{L_{p,w}\left(  {\mathbb{R}^{n}}\right)
}=\left(
%TCIMACRO{\dint \limits_{B\left(  x,r\right)  }}%
%BeginExpansion
{\displaystyle \int \limits_{B\left(  x,r\right)  }}
%EndExpansion
|f(y)|^{p}w(y)dy\right)  ^{\frac{1}{p}}.
\]
Furthermore, by $WM_{p,\varphi}(w)\equiv WM_{p,\varphi}({\mathbb{R}^{n}},w) $
we denote the weak generalized weighted Morrey space of all classes of
functions $f\in WL_{p,w}^{loc}({\mathbb{R}^{n}})$ for which
\[
\Vert f\Vert_{WM_{p,\varphi}(w)}=\sup \limits_{x\in{\mathbb{R}^{n}},r>0}%
\varphi(x,r)^{-1}\,w(B(x,r))^{-\frac{1}{p}}\, \Vert f\Vert_{WL_{p,w}%
(B(x,r))}<\infty,
\]
where $WL_{p,w}(B(x,r))$ denotes the weighted weak $WL_{p,w}$-space of
measurable functions $f$ for which%
\[
\Vert f\Vert_{WL_{p,w}\left(  B\left(  x,r\right)  \right)  }\equiv \Vert
f\chi_{B\left(  x,r\right)  }\Vert_{WL_{p,w}\left(  {\mathbb{R}^{n}}\right)
}=\sup \limits_{t>0}tw\left(  \left \{  y\in B\left(  x,r\right)  {:}\left \vert
f\left(  y\right)  \right \vert >t\right \}  \right)  ^{\frac{1}{p}}<\infty.
\]

\end{definition}

\begin{remark}
$(1)~$ If $w\equiv1$, then $M_{p,\varphi}(1)=M_{p,\varphi}$ is the generalized
Morrey space.

$(2)~$ If $\varphi(x,r)\equiv w(B(x,r))^{\frac{\kappa-1}{p}}$, $0<\kappa<1 $,
then $M_{p,\varphi}(w)=L_{p,\kappa}(w)$ is the weighted Morrey space.

$(3)~$ If $\varphi(x,r)\equiv \nu(B(x,r))^{\frac{\kappa}{p}}w(B(x,r))^{-\frac
{1}{p}}$, $0<\kappa<1$, then $M_{p,\varphi}(w)=L_{p,\kappa}(\nu,w)$ is the two
weighted Morrey space.

$(4)~$ If $w\equiv1$ and $\varphi(x,r)=r^{\frac{\lambda-n}{p}}$ with
$0\leq \lambda \leq n$, then $M_{p,\varphi}(1)=M_{p,\lambda}$ is the classical
Morrey space and $WM_{p,\varphi}(1)=WM_{p,\lambda}$ is the weak Morrey space.

$(5)~$ If $\varphi(x,r)\equiv w(B(x,r))^{-\frac{1}{p}}$, then $M_{p,\varphi
}(w)=L_{p}(w)$ is the weighted Lebesgue space.
\end{remark}

Inspired by the above results, in this paper we are interested in the
boundedness of sublinear operators with rough kernel on generalized weighted
Morrey spaces and give bounded mean oscillation {space} estimates for their commutators.

In this section we prove boundedness of the operator $T_{\Omega}$ satisfying
(\ref{e1}) on the generalized weighted Morrey spaces $M_{p,\varphi}\left(
w\right)  $ by using the following main Lemma \ref{lemma2}.

\begin{theorem}
\label{teo1}$\left(  \text{see \cite{LuDingY}}\right)  $ Let $\Omega \in
L_{s}(S^{n-1})$, $s>1$, be homogeneous of degree zero, and $1\leq p<\infty$.
If $p$, $q$ and the weight function $w$ satisfy one of the following statements:

$\left(  i\right)  $ $s^{\prime}\leq p<\infty$, $p\neq1$ and $w\in A_{\frac
{p}{s^{\prime}}}$;

$\left(  ii\right)  $ $1<p\leq s$, $p\neq \infty$ and $w^{1-p^{\prime}}\in
A_{\frac{p^{\prime}}{s^{\prime}}}$;

$\left(  iii\right)  $ $1<p<\infty$, and $w^{s^{\prime}}\in A_{p}$,

then $\overline{T}_{\Omega}$ is bounded on $L_{p}(w)$.
\end{theorem}

We first prove the following main Lemma \ref{lemma2}.

\begin{lemma}
\label{lemma2}(Our main Lemma) Let $\Omega \in L_{s}(S^{n-1})$, $s>1$, be
homogeneous of degree zero, and $1\leq p<\infty$. Let $T_{\Omega}$ be a
sublinear operator satisfying condition (\ref{e1}), bounded on $L_{p}(w)$ for
$p>1$ and bounded from $L_{1}(w)$ to $WL_{1}(w)$.

If $p>1$, $s^{\prime}\leq p$ and $w\in A_{\frac{p}{s^{\prime}}}$, then the
inequality%
\begin{equation}
\left \Vert T_{\Omega}f\right \Vert _{L_{p,w}\left(  B\left(  x_{0},r\right)
\right)  }\lesssim w\left(  B\left(  x_{0},r\right)  \right)  ^{\frac{1}{p}%
}\int \limits_{2r}^{\infty}\left \Vert f\right \Vert _{L_{p,w}\left(  B\left(
x_{0},t\right)  \right)  }w\left(  B\left(  x_{0},t\right)  \right)
^{-\frac{1}{p}}\frac{dt}{t}\label{40}%
\end{equation}
holds for any ball $B\left(  x_{0},r\right)  $ and for all $f\in L_{p,w}%
^{loc}\left(  {\mathbb{R}^{n}}\right)  $.

If $p>1$, $p<s$ and $w^{1-p^{\prime}}\in A_{\frac{p^{\prime}}{s^{\prime}}} $,
then the inequality%
\[
\left \Vert T_{\Omega}f\right \Vert _{L_{p,w}\left(  B\left(  x_{0},r\right)
\right)  }\lesssim \left \Vert w\right \Vert _{L_{\frac{s}{s-p}}\left(  B\left(
x_{0},r\right)  \right)  }^{\frac{1}{p}}\int \limits_{2r}^{\infty}\left \Vert
f\right \Vert _{L_{p,w}\left(  B\left(  x_{0},t\right)  \right)  }\left \Vert
w\right \Vert _{L_{\frac{s}{s-p}}\left(  B\left(  x_{0},t\right)  \right)
}^{-\frac{1}{p}}\frac{dt}{t}%
\]
holds for any ball $B\left(  x_{0},r\right)  $ and for all $f\in L_{p,w}%
^{loc}\left(  {\mathbb{R}^{n}}\right)  $.

Moreover, for $s>1$ the inequality%
\begin{equation}
\left \Vert T_{\Omega}f\right \Vert _{WL_{1,w}\left(  B\left(  x_{0},r\right)
\right)  }\lesssim w\left(  B\left(  x_{0},r\right)  \right)  \int
\limits_{2r}^{\infty}\left \Vert f\right \Vert _{L_{1,w}\left(  B\left(
x_{0},t\right)  \right)  }w\left(  B\left(  x_{0},t\right)  \right)
^{-1}\frac{dt}{t}\label{e38}%
\end{equation}
holds for any ball $B\left(  x_{0},r\right)  $ and for all $f\in L_{1,w}%
^{loc}\left(  {\mathbb{R}^{n}}\right)  $.
\end{lemma}

\begin{proof}
For $x\in B\left(  x_{0},t\right)  $, notice that $\Omega$ is homogenous of
degree zero and $\Omega \in L_{s}(S^{n-1})$, $s>1$. Then, we obtain%
\begin{align}
\left(  \int \limits_{B\left(  x_{0},t\right)  }\left \vert \Omega \left(
x-y\right)  \right \vert ^{s}dy\right)  ^{\frac{1}{s}}  & =\left(
\int \limits_{B\left(  x-x_{0},t\right)  }\left \vert \Omega \left(  z\right)
\right \vert ^{s}dz\right)  ^{\frac{1}{s}}\nonumber \\
& \leq \left(  \int \limits_{B\left(  0,t+\left \vert x-x_{0}\right \vert \right)
}\left \vert \Omega \left(  z\right)  \right \vert ^{s}dz\right)  ^{\frac{1}{s}%
}\nonumber \\
& \leq \left(  \int \limits_{B\left(  0,2t\right)  }\left \vert \Omega \left(
z\right)  \right \vert ^{s}dz\right)  ^{\frac{1}{s}}\nonumber \\
& =\left(  \int \limits_{0}^{2t}\int \limits_{S^{n-1}}\left \vert \Omega \left(
z^{\prime}\right)  \right \vert ^{s}d\sigma \left(  z^{\prime}\right)
r^{n-1}dr\right)  ^{\frac{1}{s}}\nonumber \\
& =C\left \Vert \Omega \right \Vert _{L_{s}\left(  S^{n-1}\right)  }\left \vert
B\left(  x_{0},2t\right)  \right \vert ^{\frac{1}{s}}.\label{10}%
\end{align}
Let $1<p<\infty$, $s^{\prime}\leq p$ and $w\in A_{\frac{p}{s^{\prime}}}$. For
any $x_{0}\in{\mathbb{R}^{n}}$, set $B=B\left(  x_{0},r\right)  $ for the ball
centered at $x_{0}$ and of radius $r$ and $2B=B\left(  x_{0},2r\right)  $. We
represent $f$ as%
\begin{equation}
f=f_{1}+f_{2},\text{ \  \ }f_{1}\left(  y\right)  =f\left(  y\right)  \chi
_{2B}\left(  y\right)  ,\text{ \  \ }f_{2}\left(  y\right)  =f\left(  y\right)
\chi_{\, \! \left(  2B\right)  ^{C}}\left(  y\right)  ,\text{ \  \ }%
r>0\label{e39}%
\end{equation}
and have%
\[
\left \Vert T_{\Omega}f\right \Vert _{L_{p,w}\left(  B\right)  }\leq \left \Vert
T_{\Omega}f_{1}\right \Vert _{L_{p,w}\left(  B\right)  }+\left \Vert T_{\Omega
}f_{2}\right \Vert _{L_{p,w}\left(  B\right)  }.
\]

Since $f_{1}\in L_{p}\left(  w\right)  $, $T_{\Omega}f_{1}\in L_{p}\left(
w\right)  $ and by the boundedness of $T_{\Omega}$ on $L_{p}\left(  w\right)
$ (see Theorem \ref{teo1}) it follows that:%
\[
\left \Vert T_{\Omega}f_{1}\right \Vert _{L_{p,w}\left(  B\right)  }%
\leq \left \Vert T_{\Omega}f_{1}\right \Vert _{L_{p,w}\left(
%TCIMACRO{\U{211d} }%
%BeginExpansion
\mathbb{R}
%EndExpansion
^{n}\right)  }\leq C\left \Vert f_{1}\right \Vert _{L_{p,w}\left(
%TCIMACRO{\U{211d} }%
%BeginExpansion
\mathbb{R}
%EndExpansion
^{n}\right)  }=C\left \Vert f\right \Vert _{L_{p,w}\left(  2B\right)  },
\]
where constant $C>0$ is independent of $f$.

It is clear that $x\in B$, $y\in \left(  2B\right)  ^{C}$ implies \ $\frac
{1}{2}\left \vert x_{0}-y\right \vert \leq \left \vert x-y\right \vert \leq \frac
{3}{2}\left \vert x_{0}-y\right \vert $. We get%
\[
\left \vert T_{\Omega}f_{2}\left(  x\right)  \right \vert \leq2^{n}c_{1}%
\int \limits_{\left(  2B\right)  ^{C}}\frac{\left \vert f\left(  y\right)
\right \vert \left \vert \Omega \left(  x-y\right)  \right \vert }{\left \vert
x_{0}-y\right \vert ^{n}}dy.
\]

By the Fubini's theorem, we have%
\begin{align}
\int \limits_{\left(  2B\right)  ^{C}}\frac{\left \vert f\left(  y\right)
\right \vert \left \vert \Omega \left(  x-y\right)  \right \vert }{\left \vert
x_{0}-y\right \vert ^{n}}dy  & \approx \int \limits_{\left(  2B\right)  ^{C}%
}\left \vert f\left(  y\right)  \right \vert \left \vert \Omega \left(
x-y\right)  \right \vert \int \limits_{\left \vert x_{0}-y\right \vert }^{\infty
}\frac{dt}{t^{n+1}}dy\nonumber \\
& \approx \int \limits_{2r}^{\infty}\int \limits_{2r\leq \left \vert x_{0}%
-y\right \vert \leq t}\left \vert f\left(  y\right)  \right \vert \left \vert
\Omega \left(  x-y\right)  \right \vert dy\frac{dt}{t^{n+1}}\nonumber \\
& \lesssim \int \limits_{2r}^{\infty}\int \limits_{B\left(  x_{0},t\right)
}\left \vert f\left(  y\right)  \right \vert \left \vert \Omega \left(
x-y\right)  \right \vert dy\frac{dt}{t^{n+1}}.\label{100}%
\end{align}

Applying the H\"{o}lder's inequality and by (\ref{10}) and (\ref{3}), we get%
\begin{align}
& \int \limits_{\left(  2B\right)  ^{C}}\frac{\left \vert f\left(  y\right)
\right \vert \left \vert \Omega \left(  x-y\right)  \right \vert }{\left \vert
x_{0}-y\right \vert ^{n}}dy\nonumber \\
& \lesssim \int \limits_{2r}^{\infty}\left \Vert \Omega \left(  x-\cdot \right)
\right \Vert _{L_{s}\left(  B\left(  x_{0},t\right)  \right)  }\left \Vert
f\right \Vert _{L_{s^{\prime}}\left(  B\left(  x_{0},t\right)  \right)  }%
\frac{dt}{t^{n+1}}\nonumber \\
& \lesssim \int \limits_{2r}^{\infty}\left \Vert f\right \Vert _{L_{p,w}\left(
B\left(  x_{0},t\right)  \right)  }\left \Vert w^{-\frac{s^{\prime}}{p}%
}\right \Vert _{L_{\left(  \frac{p}{s^{\prime}}\right)  ^{\prime}}\left(
B\left(  x_{0},t\right)  \right)  }^{\frac{1}{s^{\prime}}}\left \vert B\left(
x_{0},2t\right)  \right \vert ^{\frac{1}{s}}\frac{dt}{t^{n+1}}\nonumber \\
& \lesssim \int \limits_{2r}^{\infty}\left \Vert f\right \Vert _{L_{p,w}\left(
B\left(  x_{0},t\right)  \right)  }w\left(  B\left(  x_{0},t\right)  \right)
^{-\frac{1}{p}}\left \vert B\left(  x_{0},t\right)  \right \vert ^{\frac
{1}{s^{\prime}}}\left \vert B\left(  x_{0},2t\right)  \right \vert ^{\frac{1}%
{s}}\frac{dt}{t^{n+1}}\nonumber \\
& \lesssim \int \limits_{2r}^{\infty}\left \Vert f\right \Vert _{L_{p,w}\left(
B\left(  x_{0},t\right)  \right)  }w\left(  B\left(  x_{0},t\right)  \right)
^{-\frac{1}{p}}\frac{dt}{t}.\label{11}%
\end{align}

Thus, by (\ref{11}), it follows that:%
\[
\left \vert T_{\Omega}f_{2}\left(  x\right)  \right \vert \lesssim
\int \limits_{2r}^{\infty}\left \Vert f\right \Vert _{L_{p,w}\left(  B\left(
x_{0},t\right)  \right)  }w\left(  B\left(  x_{0},t\right)  \right)
^{-\frac{1}{p}}\frac{dt}{t}.
\]

Moreover, for all $p\in \left[  1,\infty \right)  $ the inequality%
\begin{equation}
\left \Vert T_{\Omega}f_{2}\right \Vert _{L_{p,w}\left(  B\right)  }\lesssim
w\left(  B\left(  x_{0},r\right)  \right)  ^{\frac{1}{p}}\int \limits_{2r}%
^{\infty}\left \Vert f\right \Vert _{L_{p,w}\left(  B\left(  x_{0},t\right)
\right)  }w\left(  B\left(  x_{0},t\right)  \right)  ^{-\frac{1}{p}}\frac
{dt}{t}\label{e312}%
\end{equation}

is valid. Thus,%
\[
\left \Vert T_{\Omega}f\right \Vert _{L_{p,w}\left(  B\right)  }\lesssim
\left \Vert f\right \Vert _{L_{p,w}\left(  2B\right)  }+w\left(  B\left(
x_{0},r\right)  \right)  ^{\frac{1}{p}}\int \limits_{2r}^{\infty}\left \Vert
f\right \Vert _{L_{p,w}\left(  B\left(  x_{0},t\right)  \right)  }w\left(
B\left(  x_{0},t\right)  \right)  ^{-\frac{1}{p}}\frac{dt}{t}.
\]

On the other hand, it is clear that $w\in A_{\frac{p}{s^{\prime}}}$ implies
$w\in A_{p}$, by (\ref{1}) and (\ref{3}) we have%
\begin{align}
\left \Vert f\right \Vert _{L_{p,w}\left(  2B\right)  }  & \approx \left \vert
B\right \vert \left \Vert f\right \Vert _{L_{p,w}\left(  2B\right)  }%
\int \limits_{2r}^{\infty}\frac{dt}{t^{n+1}}\nonumber \\
& \lesssim \left \vert B\right \vert \int \limits_{2r}^{\infty}\left \Vert
f\right \Vert _{L_{p,w}\left(  B\left(  x_{0},t\right)  \right)  }\frac
{dt}{t^{n+1}}\nonumber \\
& \lesssim w\left(  B\left(  x_{0},r\right)  \right)  ^{\frac{1}{p}}\left \Vert
w^{-\frac{1}{p}}\right \Vert _{L_{p^{\prime}}\left(  B\right)  }\int
\limits_{2r}^{\infty}\left \Vert f\right \Vert _{L_{p,w}\left(  B\left(
x_{0},t\right)  \right)  }\frac{dt}{t^{n+1}}\nonumber \\
& \lesssim w\left(  B\left(  x_{0},r\right)  \right)  ^{\frac{1}{p}}%
\int \limits_{2r}^{\infty}\left \Vert f\right \Vert _{L_{p,w}\left(  B\left(
x_{0},t\right)  \right)  }\left \Vert w^{-\frac{1}{p}}\right \Vert
_{L_{p^{\prime}}\left(  B\left(  x_{0},t\right)  \right)  }\frac{dt}{t^{n+1}%
}\nonumber \\
& \lesssim w\left(  B\left(  x_{0},r\right)  \right)  ^{\frac{1}{p}}%
\int \limits_{2r}^{\infty}\left \Vert f\right \Vert _{L_{p,w}\left(  B\left(
x_{0},t\right)  \right)  }w\left(  B\left(  x_{0},t\right)  \right)
^{-\frac{1}{p}}\frac{dt}{t}.\label{12}%
\end{align}

By combining the above inequalities, we obtain%
\[
\left \Vert T_{\Omega}f\right \Vert _{L_{p,w}\left(  B\left(  x_{0},r\right)
\right)  }\lesssim w\left(  B\left(  x_{0},r\right)  \right)  ^{\frac{1}{p}%
}\int \limits_{2r}^{\infty}\left \Vert f\right \Vert _{L_{p,w}\left(  B\left(
x_{0},t\right)  \right)  }w\left(  B\left(  x_{0},t\right)  \right)
^{-\frac{1}{p}}\frac{dt}{t}.
\]

Let $1<p<s$ and $w^{1-p^{\prime}}\in A_{\frac{p^{\prime}}{s^{\prime}}}$.
Similarly to (\ref{10}), when $y\in B\left(  x_{0},t\right)  $, it is true
that%
\begin{equation}
\left(  \int \limits_{B\left(  x_{0},r\right)  }\left \vert \Omega \left(
x-y\right)  \right \vert ^{s}dy\right)  ^{\frac{1}{s}}\leq C\left \Vert
\Omega \right \Vert _{L_{s}\left(  S^{n-1}\right)  }\left \vert B\left(
x_{0},\frac{3}{2}t\right)  \right \vert ^{\frac{1}{s}}.\label{314}%
\end{equation}

By the Fubini's theorem, the Minkowski inequality, Lemma \ref{lemma10},
(\ref{314}) and the H\"{o}lder's inequality, respectively we get%
\begin{align*}
\left \Vert T_{\Omega}f_{2}\right \Vert _{L_{p,w}\left(  B\right)  }  &
\leq \left(  \int \limits_{B}\left \vert \int \limits_{2r}^{\infty}\int
\limits_{B\left(  x_{0},t\right)  }\left \vert f\left(  y\right)  \right \vert
\left \vert \Omega \left(  x-y\right)  \right \vert dy\frac{dt}{t^{n+1}%
}\right \vert ^{p}w\left(  x\right)  dx\right)  ^{\frac{1}{p}}\\
& \leq \int \limits_{2r}^{\infty}\int \limits_{B\left(  x_{0},t\right)
}\left \vert f\left(  y\right)  \right \vert \left \Vert \Omega \left(
\cdot-y\right)  \right \Vert _{L_{p,w}\left(  B\right)  }dy\frac{dt}{t^{n+1}}\\
& \lesssim \int \limits_{2r}^{\infty}\int \limits_{B\left(  x_{0},t\right)
}\left \vert f\left(  y\right)  \right \vert \left \Vert \Omega \left(
\cdot-y\right)  \right \Vert _{L_{s}\left(  B\right)  }\left \Vert w\right \Vert
_{L_{\left(  \frac{s}{p}\right)  ^{\prime}}\left(  B\right)  }^{\frac{1}{p}%
}dy\frac{dt}{t^{n+1}}\\
& \lesssim \left \Vert w\right \Vert _{L_{\left(  \frac{s}{p}\right)  ^{\prime}%
}\left(  B\right)  }^{\frac{1}{p}}\int \limits_{2r}^{\infty}\left \Vert
f\right \Vert _{L_{1}\left(  B\left(  x_{0},t\right)  \right)  }\left \vert
B\left(  x_{0},\frac{3}{2}t\right)  \right \vert ^{\frac{1}{s}}\frac
{dt}{t^{n+1}}\\
& \lesssim \left \Vert w\right \Vert _{L_{\frac{s}{s-p}}\left(  B\left(
x_{0},r\right)  \right)  }^{\frac{1}{p}}\int \limits_{2r}^{\infty}\left \Vert
f\right \Vert _{L_{p,w}\left(  B\left(  x_{0},t\right)  \right)  }\left \Vert
w^{-\frac{p^{\prime}}{p}}\right \Vert _{L_{1}\left(  B\left(  x_{0},t\right)
\right)  }^{\frac{1}{p^{\prime}}}\left \vert B\left(  x_{0},\frac{3}%
{2}t\right)  \right \vert ^{\frac{1}{s}}\frac{dt}{t^{n+1}}\\
& \lesssim \left \vert B\left(  x_{0},r\right)  \right \vert ^{\frac{1}{s}%
}\left \Vert w\right \Vert _{L_{\frac{s}{s-p}}\left(  B\left(  x_{0},r\right)
\right)  }^{\frac{1}{p}}\int \limits_{2r}^{\infty}\left \Vert f\right \Vert
_{L_{p,w}\left(  B\left(  x_{0},t\right)  \right)  }\left \Vert w^{1-p^{\prime
}}\right \Vert _{L_{1}\left(  B\left(  x_{0},t\right)  \right)  }^{\frac
{1}{p^{\prime}}}\left \vert B\left(  x_{0},\frac{3}{2}t\right)  \right \vert
^{\frac{1}{s}}\frac{dt}{t^{n+1}}.
\end{align*}
Applying (\ref{7}) and (\ref{9}) for $\left \Vert w^{1-p^{\prime}}\right \Vert
_{L_{1}\left(  B\left(  x_{0},t\right)  \right)  }^{\frac{1}{p^{\prime}}}$ and
$\left \Vert w\right \Vert _{L_{\frac{s}{s-p}}\left(  B\left(  x_{0},r\right)
\right)  }^{\frac{1}{p}}$, respectively we have%
\[
\left \Vert T_{\Omega}f_{2}\right \Vert _{L_{p,w}\left(  B\left(  x_{0}%
,r\right)  \right)  }\lesssim \left \Vert w\right \Vert _{L_{\frac{s}{s-p}%
}\left(  B\left(  x_{0},r\right)  \right)  }^{\frac{1}{p}}\int \limits_{2r}%
^{\infty}\left \Vert f\right \Vert _{L_{p,w}\left(  B\left(  x_{0},t\right)
\right)  }\left \Vert w\right \Vert _{L_{\frac{s}{s-p}}\left(  B\left(
x_{0},t\right)  \right)  }^{-\frac{1}{p}}\frac{dt}{t}.
\]
Therefore,%
\[
\left \Vert T_{\Omega}f\right \Vert _{L_{p,w}\left(  B\right)  }\lesssim
\left \Vert f\right \Vert _{L_{p,w}\left(  2B\right)  }+\left \Vert w\right \Vert
_{L_{\frac{s}{s-p}}\left(  B\left(  x_{0},r\right)  \right)  }^{\frac{1}{p}%
}\int \limits_{2r}^{\infty}\left \Vert f\right \Vert _{L_{p,w}\left(  B\left(
x_{0},t\right)  \right)  }\left \Vert w\right \Vert _{L_{\frac{s}{s-p}}\left(
B\left(  x_{0},t\right)  \right)  }^{-\frac{1}{p}}\frac{dt}{t}.
\]

On the other hand, we have%
\begin{align*}
\left \Vert f\right \Vert _{L_{p,w}\left(  2B\right)  }  & \approx \left \vert
B\right \vert \left \Vert f\right \Vert _{L_{p,w}\left(  2B\right)  }%
\int \limits_{2r}^{\infty}\frac{dt}{t^{n+1}}\\
& \lesssim \left \vert B\right \vert \int \limits_{2r}^{\infty}\left \Vert
f\right \Vert _{L_{p,w}\left(  B\left(  x_{0},t\right)  \right)  }\frac
{dt}{t^{n+1}}\\
& \lesssim \left \vert B\right \vert ^{\frac{1}{s}}\left \Vert w^{1-p^{\prime}%
}\right \Vert _{L_{1}\left(  B\right)  }^{\frac{1}{p^{\prime}}}\left \Vert
w\right \Vert _{L_{\frac{s}{s-p}}\left(  B\right)  }^{\frac{1}{p}}%
\int \limits_{2r}^{\infty}\left \Vert f\right \Vert _{L_{p,w}\left(  B\left(
x_{0},t\right)  \right)  }\frac{dt}{t^{n+1}}\\
& \lesssim \left \Vert w\right \Vert _{L_{\frac{s}{s-p}}\left(  B\left(
x_{0},r\right)  \right)  }^{\frac{1}{p}}\int \limits_{2r}^{\infty}\left \Vert
f\right \Vert _{L_{p,w}\left(  B\left(  x_{0},t\right)  \right)  }\left \Vert
w\right \Vert _{L_{\frac{s}{s-p}}\left(  B\left(  x_{0},t\right)  \right)
}^{-\frac{1}{p}}\frac{dt}{t}.
\end{align*}

By combining the above inequalities, we obtain%
\[
\left \Vert T_{\Omega}f\right \Vert _{L_{p,w}\left(  B\left(  x_{0},r\right)
\right)  }\lesssim \left \Vert w\right \Vert _{L_{\frac{s}{s-p}}\left(  B\left(
x_{0},r\right)  \right)  }^{\frac{1}{p}}\int \limits_{2r}^{\infty}\left \Vert
f\right \Vert _{L_{p,w}\left(  B\left(  x_{0},t\right)  \right)  }\left \Vert
w\right \Vert _{L_{\frac{s}{s-p}}\left(  B\left(  x_{0},t\right)  \right)
}^{-\frac{1}{p}}\frac{dt}{t}.
\]

Let $p=1<s\leq \infty$. From the weak $\left(  1,1\right)  $ boundedness of
$T_{\Omega}$ and (\ref{12}) it follows that:%

\begin{align}
\left \Vert T_{\Omega}f_{1}\right \Vert _{WL_{1,w}\left(  B\right)  }  &
\leq \left \Vert T_{\Omega}f_{1}\right \Vert _{WL_{1,w}\left(
%TCIMACRO{\U{211d} }%
%BeginExpansion
\mathbb{R}
%EndExpansion
^{n}\right)  }\lesssim \left \Vert f_{1}\right \Vert _{L_{1,w}\left(
%TCIMACRO{\U{211d} }%
%BeginExpansion
\mathbb{R}
%EndExpansion
^{n}\right)  }\nonumber \\
& =\left \Vert f\right \Vert _{L_{1,w}\left(  2B\right)  }\lesssim w\left(
B\left(  x_{0},r\right)  \right)  \int \limits_{2r}^{\infty}\left \Vert
f\right \Vert _{L_{1,w}\left(  B\left(  x_{0},t\right)  \right)  }w\left(
B\left(  x_{0},t\right)  \right)  ^{-1}\frac{dt}{t}.\label{315}%
\end{align}

Then from (\ref{e312}) and (\ref{315}) we get the inequality (\ref{e38}),
which completes the proof.
\end{proof}

In the following theorem we get the boundedness of the operator $T_{\Omega}$
on the generalized weighted Morrey spaces $M_{p,\varphi}\left(  w\right)  $.

\begin{theorem}
\label{teo9}$\left(  \text{Our main result}\right)  $ Let $\Omega \in
L_{s}(S^{n-1})$, $s>1$, be homogeneous of degree zero, and $1\leq p<\infty$.
Let $T_{\Omega}$ be a sublinear operator satisfying condition (\ref{e1}),
bounded on $L_{p}(w)$ for $p>1$ and bounded from $L_{1}(w)$ to $WL_{1}(w)$.
Let also, for $s^{\prime}\leq p$, $p\neq1$ and $w\in A_{\frac{p}{s^{\prime}}}%
$, the pair $(\varphi_{1},\varphi_{2})$ satisfies the condition%
\begin{equation}
\int \limits_{r}^{\infty}\frac{\operatorname*{essinf}\limits_{t<\tau<\infty
}\varphi_{1}(x,\tau)w\left(  B\left(  x,\tau \right)  \right)  ^{\frac{1}{p}}%
}{w\left(  B\left(  x,t\right)  \right)  ^{\frac{1}{p}}}\frac{dt}{t}\leq C\,
\varphi_{2}(x,r),\label{316}%
\end{equation}
and for $1<p<s$ and $w^{1-p^{\prime}}\in A_{\frac{p^{\prime}}{s^{\prime}}} $
the pair $(\varphi_{1},\varphi_{2})$ satisfies the condition%
\begin{equation}
\int \limits_{r}^{\infty}\frac{\operatorname*{essinf}\limits_{t<\tau<\infty
}\varphi_{1}(x,\tau)\left \Vert w\right \Vert _{L_{\frac{s}{s-p}}\left(
B\left(  x,r\right)  \right)  }^{\frac{1}{p}}}{\left \Vert w\right \Vert
_{L_{\frac{s}{s-p}}\left(  B\left(  x,t\right)  \right)  }^{\frac{1}{p}}}%
\frac{dt}{t}\leq C\, \varphi_{2}(x,r)\frac{w\left(  B\left(  x,r\right)
\right)  ^{\frac{1}{p}}}{\left \Vert w\right \Vert _{L_{\frac{s}{s-p}}\left(
B\left(  x,r\right)  \right)  }^{\frac{1}{p}}},\label{317}%
\end{equation}
where $C$ does not depend on $x$ and $r$.

Then the operator $T_{\Omega}$ is bounded from $M_{p,\varphi_{1}}\left(
w\right)  $ to $M_{p,\varphi_{2}}\left(  w\right)  $ for $p>1$ and from
$M_{1,\varphi_{1}}\left(  w\right)  $ to $WM_{1,\varphi_{2}}\left(  w\right)
$. Moreover, we have for $p>1$%
\begin{equation}
\left \Vert T_{\Omega}f\right \Vert _{M_{p,\varphi_{2}}\left(  w\right)
}\lesssim \left \Vert f\right \Vert _{M_{p,\varphi_{1}}\left(  w\right)
},\label{3-1}%
\end{equation}
and for $p=1$%
\begin{equation}
\left \Vert T_{\Omega}f\right \Vert _{WM_{1,\varphi_{2}}\left(  w\right)
}\lesssim \left \Vert f\right \Vert _{M_{1,\varphi_{1}}\left(  w\right)
}.\label{3-2}%
\end{equation}

\end{theorem}

\begin{proof}
since $f\in M_{p,\varphi_{1}}\left(  w\right)  $, by (\ref{4}) and the
non-decreasing, with respect to $t$, of the norm $\left \Vert f\right \Vert
_{L_{p,w}\left(  B\left(  x_{0},t\right)  \right)  }$, we get%
\begin{align*}
& \frac{\left \Vert f\right \Vert _{L_{p,w}\left(  B\left(  x_{0},t\right)
\right)  }}{\operatorname*{essinf}\limits_{0<t<\tau<\infty}\varphi_{1}%
(x_{0},\tau)w\left(  B\left(  x_{0},\tau \right)  \right)  ^{\frac{1}{p}}}\\
& \leq \operatorname*{esssup}\limits_{0<t<\tau<\infty}\frac{\left \Vert
f\right \Vert _{L_{p,w}\left(  B\left(  x_{0},t\right)  \right)  }}{\varphi
_{1}(x_{0},\tau)w\left(  B\left(  x_{0},\tau \right)  \right)  ^{\frac{1}{p}}%
}\\
& \leq \operatorname*{esssup}\limits_{0<\tau<\infty}\frac{\left \Vert
f\right \Vert _{L_{p,w}\left(  B\left(  x_{0},\tau \right)  \right)  }}%
{\varphi_{1}(x_{0},\tau)w\left(  B\left(  x_{0},\tau \right)  \right)
^{\frac{1}{p}}}\\
& \leq \left \Vert f\right \Vert _{M_{p,\varphi_{1}}\left(  w\right)  }.
\end{align*}
For $s^{\prime}\leq p<\infty$, since $(\varphi_{1},\varphi_{2})$ satisfies
(\ref{316}), we have%
\begin{align*}
& \int \limits_{r}^{\infty}\left \Vert f\right \Vert _{L_{p,w}\left(  B\left(
x_{0},t\right)  \right)  }w\left(  B\left(  x_{0},t\right)  \right)
^{-\frac{1}{p}}\frac{dt}{t}\\
& \leq \int \limits_{r}^{\infty}\frac{\left \Vert f\right \Vert _{L_{p,w}\left(
B\left(  x_{0},t\right)  \right)  }}{\operatorname*{essinf}\limits_{t<\tau
<\infty}\varphi_{1}(x_{0},\tau)w\left(  B\left(  x_{0},\tau \right)  \right)
^{\frac{1}{p}}}\frac{\operatorname*{essinf}\limits_{t<\tau<\infty}\varphi
_{1}(x_{0},\tau)w\left(  B\left(  x_{0},\tau \right)  \right)  ^{\frac{1}{p}}%
}{w\left(  B\left(  x_{0},t\right)  \right)  ^{\frac{1}{p}}}\frac{dt}{t}\\
& \leq C\left \Vert f\right \Vert _{M_{p,\varphi_{1}}\left(  w\right)  }%
\int \limits_{r}^{\infty}\frac{\operatorname*{essinf}\limits_{t<\tau<\infty
}\varphi_{1}(x_{0},\tau)w\left(  B\left(  x_{0},\tau \right)  \right)
^{\frac{1}{p}}}{w\left(  B\left(  x_{0},t\right)  \right)  ^{\frac{1}{p}}%
}\frac{dt}{t}\\
& \leq C\left \Vert f\right \Vert _{M_{p,\varphi_{1}}\left(  w\right)  }%
\varphi_{2}(x_{0},r).
\end{align*}
Then by (\ref{40}), we get%
\begin{align*}
\left \Vert T_{\Omega}f\right \Vert _{M_{p,\varphi_{2}}\left(  w\right)  }  &
=\sup_{x_{0}\in{\mathbb{R}^{n},}r>0}\varphi_{2}\left(  x_{0},r\right)
^{-1}w\left(  B\left(  x_{0},r\right)  \right)  ^{-\frac{1}{p}}\left \Vert
T_{\Omega}f\right \Vert _{L_{p,w}\left(  B\left(  x_{0},r\right)  \right)  }\\
& \leq C\sup_{x_{0}\in{\mathbb{R}^{n},}r>0}\varphi_{2}\left(  x_{0},r\right)
^{-1}\int \limits_{r}^{\infty}\left \Vert f\right \Vert _{L_{p,w}\left(  B\left(
x_{0},t\right)  \right)  }w\left(  B\left(  x_{0},t\right)  \right)
^{-\frac{1}{p}}\frac{dt}{t}\\
& \leq C\left \Vert f\right \Vert _{M_{p,\varphi_{1}}\left(  w\right)  }.
\end{align*}
For the case of $1\leq p<s$, we can also use the same method, so we omit the
details. This completes the proof of Theorem \ref{teo9}.
\end{proof}

Let $f\in L_{1}^{loc}\left(  {\mathbb{R}^{n}}\right)  $. The rough
Hardy-Littlewood maximal operator $M_{\Omega}$ is defined by%

\[
M_{\Omega}f\left(  x\right)  =\sup_{t>0}\frac{1}{\left \vert B\left(
x,t\right)  \right \vert }\int \limits_{B\left(  x,t\right)  }\left \vert
\Omega \left(  x-y\right)  \right \vert \left \vert f\left(  y\right)
\right \vert dy\text{.}%
\]

Then we can get the following corollary.

\begin{corollary}
Let $1\leq p<\infty$, $\Omega \in L_{s}\left(  S^{n-1}\right)  $, $s>1$, be
homogeneous of degree zero. For $s^{\prime}\leq p$, $p\neq1$ and $w\in
A_{\frac{p}{s^{\prime}}}$, the pair $\left(  \varphi_{1},\varphi_{2}\right)  $
satisfies condition (\ref{316}) and for $1<$ $p<s$ and $w^{1-p^{\prime}}\in
A_{\frac{p^{\prime}}{s^{\prime}}}$ the pair $\left(  \varphi_{1},\varphi
_{2}\right)  $ satisfies condition (\ref{317}). Then the operators $M_{\Omega
}$ and $\overline{T}_{\Omega}$ are bounded from $M_{p,\varphi_{1}}\left(
w\right)  $ to $M_{p,\varphi_{2}}\left(  w\right)  $ for $p>1$ and from
$M_{1,\varphi_{1}}\left(  w\right)  $ to $WM_{1,\varphi_{2}}\left(  w\right)
$.
\end{corollary}

In the case of $w=1$ from Theorem \ref{teo9}, we get

\begin{corollary}
$\left(  \text{see \cite{BGGS, Gurbuz}}\right)  $ Let $\Omega \in L_{s}%
(S^{n-1})$, $s>1$, be homogeneous of degree zero, and $1\leq p<\infty$. Let
$T_{\Omega}$ be a sublinear operator satisfying condition (\ref{e1}), bounded
on $L_{p}({\mathbb{R}^{n}})$ for $p>1$, and bounded from $L_{1}({\mathbb{R}%
^{n}})$ to $WL_{1}({\mathbb{R}^{n}})$. Let also, for $s^{\prime}\leq p$,
$p\neq1$, the pair $(\varphi_{1},\varphi_{2})$ satisfies the condition%
\[
\int \limits_{r}^{\infty}\frac{\operatorname*{essinf}\limits_{t<\tau<\infty
}\varphi_{1}(x,\tau)\tau^{\frac{n}{p}}}{t^{\frac{n}{p}+1}}dt\leq C\,
\varphi_{2}(x,r),
\]
and for $1<p<s$ the pair $(\varphi_{1},\varphi_{2})$ satisfies the condition%
\[
\int \limits_{r}^{\infty}\frac{\operatorname*{essinf}\limits_{t<\tau<\infty
}\varphi_{1}(x,\tau)\tau^{\frac{n}{p}}}{t^{\frac{n}{p}-\frac{n}{s}+1}}dt\leq
C\, \varphi_{2}(x,r)r^{\frac{n}{s}},
\]
where $C$ does not depend on $x$ and $r$.

Then the operator $T_{\Omega}$ is bounded from $M_{p,\varphi_{1}}$ to
$M_{p,\varphi_{2}}$ for $p>1$ and from $M_{1,\varphi_{1}}$ to $WM_{1,\varphi
_{2}}$. Moreover, we have for $p>1$%
\[
\left \Vert T_{\Omega}f\right \Vert _{M_{p,\varphi_{2}}}\lesssim \left \Vert
f\right \Vert _{M_{p,\varphi_{1}}},
\]

and for $p=1$%
\[
\left \Vert T_{\Omega}f\right \Vert _{WM_{1,\varphi_{2}}}\lesssim \left \Vert
f\right \Vert _{M_{1,\varphi_{1}}}.
\]

\end{corollary}

In the case of $\varphi_{1}\left(  x,r\right)  =\varphi_{2}\left(  x,r\right)
\equiv w\left(  B\left(  x,r\right)  \right)  ^{^{\frac{\kappa-1}{p}}}$ from
Theorem \ref{teo9} we get the following new result.

\begin{corollary}
Let $1\leq p<\infty$, $\Omega \in L_{s}\left(  S^{n-1}\right)  $, $s>1$, be
homogeneous of degree zero and $0<\kappa<1$. Let also $T_{\Omega}$ be a
sublinear operator satisfying condition (\ref{e1}), bounded on $L_{p}(w)$ for
$p>1$ and bounded from $L_{1}(w)$ to $WL_{1}(w)$. For $s^{\prime}\leq p$,
$p\neq1$ and $w\in A_{\frac{p}{s^{\prime}}}$ or $1<$ $p<s$ and $w^{1-p^{\prime
}}\in A_{\frac{p^{\prime}}{s^{\prime}}}$, the operator $T_{\Omega}$ is bounded
on the weighted Morrey spaces $L_{p,\kappa}(w)$ for $p>1$ and bounded from
$L_{1,\kappa}(w)$ to $WL_{1,\kappa}(w)$.
\end{corollary}

When $\Omega \equiv1$, from Theorem \ref{teo9} we get

\begin{corollary}
\label{corollary 5}Let $1\leq p<\infty$, $w\in A_{p}$ and the pair
$(\varphi_{1},\varphi_{2})$ satisfies condition (\ref{316}). Let also $T$ be a
sublinear operator satisfying condition (\ref{e1}), bounded on $L_{p}(w) $ for
$p>1$ and bounded from $L_{1}(w)$ to $WL_{1}(w)$. Then the operator $T$ is
bounded from $M_{p,\varphi_{1}}\left(  w\right)  $ to $M_{p,\varphi_{2}%
}\left(  w\right)  $ for $p>1$ and from $M_{1,\varphi_{1}}\left(  w\right)  $
to $WM_{1,\varphi_{2}}\left(  w\right)  $.
\end{corollary}

\begin{remark}
Corollary \ref{corollary 5} has been proved in \cite{Karaman}.
\end{remark}

When $\Omega \equiv1$, In the case of $\varphi_{1}\left(  x,r\right)
=\varphi_{2}\left(  x,r\right)  \equiv w\left(  B\left(  x,r\right)  \right)
^{^{\frac{\kappa-1}{p}}}$ from Theorem \ref{teo9} we get the following new result.

\begin{corollary}
\label{corollary 4}$1\leq p<\infty$, $0<\kappa<1$ and $w\in A_{p}$. Let also
$T$ be a sublinear operator satisfying condition (\ref{e1}), bounded on
$L_{p}(w)$ for $p>1$ and bounded from $L_{1}(w)$ to $WL_{1}(w)$. Then the
operator $T$ is bounded on the weighted Morrey spaces $L_{p,\kappa}(w)$ for
$p>1$ and bounded from $L_{1,\kappa}(w)$ to $WL_{1,\kappa}(w)$.
\end{corollary}

\begin{remark}
Note that, from Corollary \ref{corollary 4}, we get Theorem \ref{teo9*}.
\end{remark}

\section{Commutators of sublinear operators with rough kernel generated by
{Calder\'{o}n-Zygmund type operators} on the generalized weighted Morrey
spaces $M_{p,\varphi}\left(  w\right)  $}

In this section we prove the boundedness of the operator $T_{\Omega,b}$
satisfying condition (\ref{e2}) with $b\in BMO\left(  {\mathbb{R}^{n}}\right)
$ on the generalized weighted Morrey spaces $M_{p,\varphi}\left(  w\right)  $
by using the following main Lemma \ref{Lemma 5}.

Let us recall the definition of the space of $BMO({\mathbb{R}^{n}})$.

\begin{definition}
Suppose that $b\in L_{1}^{loc}({\mathbb{R}^{n}})$, let
\[
\Vert b\Vert_{\ast}=\sup_{x\in{\mathbb{R}^{n}},r>0}\frac{1}{|B(x,r)|}%
%TCIMACRO{\dint \limits_{B(x,r)}}%
%BeginExpansion
{\displaystyle \int \limits_{B(x,r)}}
%EndExpansion
|b(y)-b_{B(x,r)}|dy<\infty,
\]
where
\[
b_{B(x,r)}=\frac{1}{|B(x,r)|}%
%TCIMACRO{\dint \limits_{B(x,r)}}%
%BeginExpansion
{\displaystyle \int \limits_{B(x,r)}}
%EndExpansion
b(y)dy.
\]

Define
\[
BMO({\mathbb{R}^{n}})=\{b\in L_{1}^{loc}({\mathbb{R}^{n}})~:~\Vert
b\Vert_{\ast}<\infty \}.
\]

If one regards two functions whose difference is a constant as one, then the
space $BMO({\mathbb{R}^{n}})$ is a Banach space with respect to norm
$\Vert \cdot \Vert_{\ast}$.
\end{definition}

An early work about $BMO({\mathbb{R}^{n}})$ space can be attributed to John
and Nirenberg \cite{John-Nirenberg}. For $1<p<\infty$, there is a close
relation between $BMO({\mathbb{R}^{n}})$ and $A_{p}$ weights:%
\[
BMO({\mathbb{R}^{n}})=\left \{  \alpha \log w:w\in A_{p}\text{, }\alpha
\geq0\right \}  .
\]

Let $T$ be a linear operator. For a locally integrable function $b$ on
${\mathbb{R}^{n}}$, we define the commutator $[b,T]$ by
\[
\lbrack b,T]f(x)=b(x)\,Tf(x)-T(bf)(x)
\]
for any suitable function $f$. Since $L_{\infty}({\mathbb{R}^{n}%
})\varsubsetneq BMO({\mathbb{R}^{n}})$, the boundedness of $[b,T]$ is worse
than $T$ (e.g., the singularity; see also \cite{Perez}). Therefore, many
authors want to know whether $[b,T]$ shares the similar boundedness with $T$.
There are a lot of articles that deal with the topic of commutators of
different operators with $BMO$ functions on Lebesgue spaces. The first results
for this commutator have been obtained by Coifman et al. \cite{CRW} in their
study of certain factorization theorems for generalized Hardy spaces. Let
$\overline{T}$ be a C--Z operator. A well known result of Coifman et al.
\cite{CRW} states that when $K\left(  x\right)  =\frac{\Omega \left(
x^{\prime}\right)  }{\left \vert x\right \vert ^{n}}$ and $\Omega$ is smooth,
the commutator $[b,\overline{T}]f=b\, \overline{T}f-\overline{T}(bf)$ is
bounded on $L_{p}({\mathbb{R}^{n}})$, $1<p<\infty$, if and only if $b\in
BMO({\mathbb{R}^{n}})$. The commutators of C--Z operator play an important
role in studying the regularity of solutions of elliptic, parabolic and
ultraparabolic partial differential equations of second order (see, for
example, \cite{ChFraL1, ChFraL2, FazRag2, Poli-Ragu}). The boundedness of the
commutator has been generalized to other contexts and important applications
to some non-linear PDEs have been given by Coifman et al. \cite{CLMS}.

The following lemmas about $BMO({\mathbb{R}^{n}})$ functions will help us to
prove Lemma \ref{Lemma 5} and Theorem \ref{teo15}.

\begin{lemma}
$\left(  \text{see }\left[  \text{\cite{MuckWh}},\text{ Theorem5, page
236}\right]  \right)  $ Let $w\in A_{\infty}$. Then the norm of $BMO(w)$ is
equivalent to the norm of $BMO({\mathbb{R}^{n}})$, where
\[
BMO(w)=\{b~:~\Vert b\Vert_{\ast,w}=\sup_{x\in{\mathbb{R}^{n}},r>0}\frac
{1}{w(B(x,r))}%
%TCIMACRO{\dint \limits_{B(x,r)}}%
%BeginExpansion
{\displaystyle \int \limits_{B(x,r)}}
%EndExpansion
|b(y)-b_{B(x,r),w}|w(y)dy<\infty \}
\]
and
\[
b_{B(x,r),w}=\frac{1}{w(B(x,r))}%
%TCIMACRO{\dint \limits_{B(x,r)}}%
%BeginExpansion
{\displaystyle \int \limits_{B(x,r)}}
%EndExpansion
b(y)w(y)dy.
\]

\end{lemma}

\begin{remark}
$(1)~$ The John-Nirenberg inequality : there are constants $C_{1}$, $C_{2}>0$,
such that for all $b\in BMO({\mathbb{R}^{n}})$ and $\beta>0$
\[
\left \vert \left \{  x\in B\,:\,|b(x)-b_{B}|>\beta \right \}  \right \vert \leq
C_{1}|B|e^{-C_{2}\beta/\Vert b\Vert_{\ast}},~~~\forall B\subset{\mathbb{R}%
^{n}}.
\]

$(2)~$ For $1<p<\infty$ the John-Nirenberg inequality implies that
\begin{equation}
\Vert b\Vert_{\ast}\thickapprox \sup_{B}\left(  \frac{1}{|B|}%
%TCIMACRO{\dint \limits_{B}}%
%BeginExpansion
{\displaystyle \int \limits_{B}}
%EndExpansion
|b(y)-b_{B}|^{p}dy\right)  ^{\frac{1}{p}}\label{5.1}%
\end{equation}
and for $1\leq p<\infty$ and $w\in A_{\infty}$
\begin{equation}
\Vert b\Vert_{\ast}\thickapprox \sup_{B}\left(  \frac{1}{w(B)}%
%TCIMACRO{\dint \limits_{B}}%
%BeginExpansion
{\displaystyle \int \limits_{B}}
%EndExpansion
|b(y)-b_{B}|^{p}w(y)dy\right)  ^{\frac{1}{p}}.\label{5.2}%
\end{equation}

Indeed, from the John-Nirenberg inequality and using Lemma \ref{lemma100} (3),
we get
\[
w(\{x\in B\,:\,|b(x)-b_{B}|>\beta \})\leq Cw(B)e^{-C_{2}\beta \delta/\Vert
b\Vert_{\ast}}%
\]
for some $\delta>0$. Hence, this inequality implies that%
\begin{align*}%
%TCIMACRO{\dint \limits_{B}}%
%BeginExpansion
{\displaystyle \int \limits_{B}}
%EndExpansion
|b(y)-b_{B}|^{p}w(y)dy &  =p%
%TCIMACRO{\dint \limits_{0}^{\infty}}%
%BeginExpansion
{\displaystyle \int \limits_{0}^{\infty}}
%EndExpansion
\beta^{p-1}w(\{x\in B\,:\,|b(x)-b_{B}|>\beta \})d\beta \\
&  \leq C\,w(B)%
%TCIMACRO{\dint \limits_{0}^{\infty}}%
%BeginExpansion
{\displaystyle \int \limits_{0}^{\infty}}
%EndExpansion
\, \beta^{p-1}e^{-C_{2}\beta \delta/\Vert b\Vert_{\ast}}\,d\beta \\
&  =Cw(B)\Vert b\Vert_{\ast}^{p}.
\end{align*}

To prove that required equivalence we also need to have the right hand
inequality, which is easily obtained using the H\"{o}lder's inequality, then
we get (\ref{5.2}). Note that (\ref{5.1}) follows from (\ref{5.2}) in the case
of $w\equiv1$.

$(3)~~$ Let $b\in BMO({\mathbb{R}^{n}})$. Then there is a constant $C>0$ such
that
\begin{equation}
\left \vert b_{B(x,r)}-b_{B(x,t)}\right \vert \leq C\Vert b\Vert_{\ast}\ln
\frac{t}{r}~\text{for}~0<2r<t,\label{5.3}%
\end{equation}
where $C$ is independent of $b$, $x$, $r$ and $t$.
\end{remark}

\begin{lemma}
$\left[  \text{\cite{Grafakos}},\text{ Proposition 7.1.2}\right]  $ $\left(
\text{see also }\left[  \text{\cite{MuckWh}, Theorem 5}\right]  \right)  $ Let
$w\in A_{\infty}$ and $1<p<\infty$. Then the following statements are equivalent:

$(1)~\Vert b\Vert_{\ast}\thickapprox \sup \limits_{B}\left(  \frac{1}{|B|}%
%TCIMACRO{\dint \limits_{B}}%
%BeginExpansion
{\displaystyle \int \limits_{B}}
%EndExpansion
|b(y)-b_{B}|^{p}dy\right)  ^{\frac{1}{p}}$,

$(2)~$ $\Vert b\Vert_{\ast}\thickapprox \sup \limits_{B}\inf \limits_{a\in%
%TCIMACRO{\U{211d} }%
%BeginExpansion
\mathbb{R}
%EndExpansion
}\frac{1}{|B|}%
%TCIMACRO{\dint \limits_{B}}%
%BeginExpansion
{\displaystyle \int \limits_{B}}
%EndExpansion
|b(y)-a|dy$,

$(3)~$ $\Vert b\Vert_{\ast,w}=\sup \limits_{B}\frac{1}{w(B)}%
%TCIMACRO{\dint \limits_{B}}%
%BeginExpansion
{\displaystyle \int \limits_{B}}
%EndExpansion
|b(y)-b_{B,w}|w(y)dy$.
\end{lemma}

The following lemma has been proved in \cite{Karaman}.

\begin{lemma}
\label{lemma0}$\left(  \text{see \cite{Karaman}}\right)  $ The following
statements hold.

$i)$ Let $w\in A_{\infty}$ and $b$ be a function in $BMO({\mathbb{R}^{n}})$.
Let also $1\leq p<\infty$, $x\in{\mathbb{R}^{n}}$, and $r_{1},r_{2}>0$. Then
\[
\left(  \frac{1}{w(B(x,r_{1}))}\int \limits_{B(x,r_{1})}|b(y)-b_{B(x,r_{2}%
),w}|^{p}w(y)dy\right)  ^{\frac{1}{p}}\leq C[w]_{A_{\infty}}^{2^{n}}\, \left(
1+\left \vert \ln \frac{r_{1}}{r_{2}}\right \vert \right)  \Vert b\Vert_{\ast},
\]
where $C>0$ is independent of $f$, $w$, $x$, $r_{1}$ and $r_{2}$.

$ii)$ Let $w\in A_{p}$ and $b$ be a function in $BMO({\mathbb{R}^{n}})$. Let
also $1<p<\infty$, $x\in{\mathbb{R}^{n}}$, and $r_{1},r_{2}>0$. Then
\[
\left(  \frac{1}{w^{1-p^{\prime}}(B(x,r_{1}))}\int \limits_{B(x,r_{1}%
)}|b(y)-b_{B(x,r_{2}),w}|^{p^{\prime}}w(y)^{1-p^{\prime}}dy\right)  ^{\frac
{1}{p^{\prime}}}\leq C[w]_{A_{p}}^{\frac{1}{p}}\, \left(  1+\left \vert
\ln \frac{r_{1}}{r_{2}}\right \vert \right)  \Vert b\Vert_{\ast},
\]
where $C>0$ is independent of $f$, $w$, $x$, $r_{1}$ and $r_{2}$.
\end{lemma}

From Lemma \ref{lemma0} we get the following corollary, which has ben proved
in \cite{LinLu} for $w\equiv1$.

\begin{corollary}
Let $b$ be a function in $BMO({\mathbb{R}^{n}})$. Let also $1\leq p<\infty$,
$x\in{\mathbb{R}^{n}}$, and $r_{1},r_{2}>0$. Then
\[
\left(  \frac{1}{|B(x,r_{1})|}%
%TCIMACRO{\dint \limits_{B(x,r_{1})}}%
%BeginExpansion
{\displaystyle \int \limits_{B(x,r_{1})}}
%EndExpansion
|b(y)-b_{B(x,r_{2})}|^{p}dy\right)  ^{\frac{1}{p}}\leq C\left(  1+\left \vert
\ln \frac{r_{1}}{r_{2}}\right \vert \right)  \Vert b\Vert_{\ast},
\]
where $C>0$ is independent of $f$, $x$, $r_{1}$ and $r_{2}$.
\end{corollary}

\begin{theorem}
\label{teo2}$\left(  \text{see \cite{LuDingY}}\right)  $ Let $\Omega \in
L_{s}(S^{n-1})$, $s>1$, be homogeneous of degree zero, and $1<p<\infty$. If
$b\in BMO\left(  {\mathbb{R}^{n}}\right)  $ and $p$, $q$, $w$ satisfy one of
the following conditions, then $[b,\overline{T}_{\Omega}]$ is bounded on
$L_{p}(w)$:

$\left(  i\right)  $ $s^{\prime}\leq p<\infty$, $p\neq1$ and $w\in A_{\frac
{p}{s^{\prime}}}$;

$\left(  ii\right)  $ $1<p\leq s$, $p\neq \infty$ and $w^{1-p^{\prime}}\in
A_{\frac{p^{\prime}}{s^{\prime}}}$;

$\left(  iii\right)  $ $1<p<\infty$, and $w^{s^{\prime}}\in A_{p}$.
\end{theorem}

As in the proof of Theorem \ref{teo9}, it suffices to prove the following main
Lemma \ref{Lemma 5}.

\begin{lemma}
\label{Lemma 5}$\left(  \text{Our main Lemma}\right)  $ Let $\Omega \in
L_{s}(S^{n-1})$, $s>1$, be homogeneous of degree zero. Let $1<p<\infty$, $b\in
BMO\left(  {\mathbb{R}^{n}}\right)  $, and $T_{\Omega,b}$ is a sublinear
operator satisfying condition (\ref{e2}), bounded on $L_{p}(w)$. Then, for
$s^{\prime}\leq p$ and $w\in A_{\frac{p}{s^{\prime}}}$ the inequality
\begin{equation}
\Vert T_{\Omega,b}f\Vert_{L_{p,w}(B(x_{0},r))}\lesssim \Vert b\Vert_{\ast
}\,w\left(  B\left(  x_{0},r\right)  \right)  ^{\frac{1}{p}}\int
\limits_{2r}^{\infty}\left(  1+\ln \frac{t}{r}\right)  \Vert f\Vert
_{L_{p,w}(B(x_{0},t))}w\left(  B\left(  x_{0},t\right)  \right)  ^{-\frac
{1}{p}}\frac{dt}{t}\label{40*}%
\end{equation}
holds for any ball $B(x_{0},r)$ and for all $f\in L_{p,w}^{loc}({\mathbb{R}%
^{n}})$.

Also, for $p<s$ and $w^{1-p^{\prime}}\in$ $A_{\frac{p^{\prime}}{s^{\prime}}}$
the inequality
\[
\Vert T_{\Omega,b}f\Vert_{L_{p,w}(B(x_{0},r))}\lesssim \Vert b\Vert_{\ast}\,
\left \Vert w\right \Vert _{L_{\frac{s}{s-p}}\left(  B\left(  x_{0},r\right)
\right)  }^{\frac{1}{p}}\int \limits_{2r}^{\infty}\left(  1+\ln \frac{t}%
{r}\right)  \Vert f\Vert_{L_{p,w}(B(x_{0},t))}\left \Vert w\right \Vert
_{L_{\frac{s}{s-p}}\left(  B\left(  x_{0},t\right)  \right)  }^{-\frac{1}{p}%
}\frac{dt}{t}%
\]
holds for any ball $B(x_{0},r)$ and for all $f\in L_{p,w}^{loc}({\mathbb{R}%
^{n}})$.
\end{lemma}

\begin{proof}
Let $1<p<\infty$ and $b\in BMO\left(  {\mathbb{R}^{n}}\right)  $. As in the
proof of Lemma \ref{lemma2}, we represent function $f$ in form (\ref{e39}) and
have%
\[
\left \Vert T_{\Omega,b}f\right \Vert _{L_{p,w}\left(  B\right)  }\leq \left \Vert
T_{\Omega,b}f_{1}\right \Vert _{L_{p,w}\left(  B\right)  }+\left \Vert
T_{\Omega,b}f_{2}\right \Vert _{L_{p,w}\left(  B\right)  }.
\]
For $s^{\prime}\leq p$ and $w\in A_{\frac{p}{s^{\prime}}}$, from the
boundedness of $T_{\Omega,b}$ on $L_{p}(w)$ (see Theorem \ref{teo2}) it
follows that:%
\begin{align*}
\left \Vert T_{\Omega,b}f_{1}\right \Vert _{L_{p,w}\left(  B\right)  }  &
\leq \left \Vert T_{\Omega,b}f_{1}\right \Vert _{L_{p,w}\left(  {\mathbb{R}^{n}%
}\right)  }\\
& \lesssim \left \Vert b\right \Vert _{\ast}\left \Vert f_{1}\right \Vert
_{L_{p,w}\left(  {\mathbb{R}^{n}}\right)  }=\left \Vert b\right \Vert _{\ast
}\left \Vert f\right \Vert _{L_{p,w}\left(  2B\right)  }.
\end{align*}
It is known that $x\in B$, $y\in \left(  2B\right)  ^{C}$, which implies
$\frac{1}{2}\left \vert x_{0}-y\right \vert \leq \left \vert x-y\right \vert
\leq \frac{3}{2}\left \vert x_{0}-y\right \vert $. Then for $x\in B$, we have%
\begin{align*}
\left \vert T_{\Omega,b}f_{2}\left(  x\right)  \right \vert  & \lesssim
\int \limits_{{\mathbb{R}^{n}}}\frac{\left \vert \Omega \left(  x-y\right)
\right \vert }{\left \vert x-y\right \vert ^{n}}\left \vert b\left(  y\right)
-b\left(  x\right)  \right \vert \left \vert f\left(  y\right)  \right \vert dy\\
& \approx \int \limits_{\left(  2B\right)  ^{C}}\frac{\left \vert \Omega \left(
x-y\right)  \right \vert }{\left \vert x_{0}-y\right \vert ^{n}}\left \vert
b\left(  y\right)  -b\left(  x\right)  \right \vert \left \vert f\left(
y\right)  \right \vert dy.
\end{align*}

Hence, we get%
\begin{align*}
\left \Vert T_{\Omega,b}f_{2}\right \Vert _{L_{p,w}\left(  B\right)  }  &
\lesssim \left(  \int \limits_{B}\left(  \int \limits_{\left(  2B\right)  ^{C}%
}\frac{\left \vert \Omega \left(  x-y\right)  \right \vert }{\left \vert
x_{0}-y\right \vert ^{n}}\left \vert b\left(  y\right)  -b\left(  x\right)
\right \vert \left \vert f\left(  y\right)  \right \vert dy\right)  ^{p}w\left(
x\right)  dx\right)  ^{\frac{1}{p}}\\
& \lesssim \left(  \int \limits_{B}\left(  \int \limits_{\left(  2B\right)  ^{C}%
}\frac{\left \vert \Omega \left(  x-y\right)  \right \vert }{\left \vert
x_{0}-y\right \vert ^{n}}\left \vert b\left(  y\right)  -b_{B,w}\right \vert
\left \vert f\left(  y\right)  \right \vert dy\right)  ^{p}w\left(  x\right)
dx\right)  ^{\frac{1}{p}}\\
& +\left(  \int \limits_{B}\left(  \int \limits_{\left(  2B\right)  ^{C}}%
\frac{\left \vert \Omega \left(  x-y\right)  \right \vert }{\left \vert
x_{0}-y\right \vert ^{n}}\left \vert b\left(  x\right)  -b_{B,w}\right \vert
\left \vert f\left(  y\right)  \right \vert dy\right)  ^{p}w\left(  x\right)
dx\right)  ^{\frac{1}{p}}\\
& =J_{1}+J_{2}.
\end{align*}
We have the following estimation of $J_{1}$. When $s^{\prime}\leq p$, by the
Fubini's theorem%
\begin{align*}
J_{1}  & \approx w\left(  B\left(  x_{0},r\right)  \right)  ^{\frac{1}{p}}%
\int \limits_{\left(  2B\right)  ^{C}}\frac{\left \vert \Omega \left(
x-y\right)  \right \vert }{\left \vert x_{0}-y\right \vert ^{n}}\left \vert
b\left(  y\right)  -b_{B,w}\right \vert \left \vert f\left(  y\right)
\right \vert dy\\
& \approx w\left(  B\left(  x_{0},r\right)  \right)  ^{\frac{1}{p}}%
\int \limits_{\left(  2B\right)  ^{C}}\left \vert \Omega \left(  x-y\right)
\right \vert \left \vert b\left(  y\right)  -b_{B,w}\right \vert \left \vert
f\left(  y\right)  \right \vert \int \limits_{\left \vert x_{0}-y\right \vert
}^{\infty}\frac{dt}{t^{n+1}}dy\\
& \approx w\left(  B\left(  x_{0},r\right)  \right)  ^{\frac{1}{p}}%
\int \limits_{2r}^{\infty}\int \limits_{2r\leq \left \vert x_{0}-y\right \vert \leq
t}\left \vert \Omega \left(  x-y\right)  \right \vert \left \vert b\left(
y\right)  -b_{B,w}\right \vert \left \vert f\left(  y\right)  \right \vert
dy\frac{dt}{t^{n+1}}\\
& \lesssim w\left(  B\left(  x_{0},r\right)  \right)  ^{\frac{1}{p}}%
\int \limits_{2r}^{\infty}\int \limits_{B\left(  x_{0},t\right)  }\left \vert
\Omega \left(  x-y\right)  \right \vert \left \vert b\left(  y\right)
-b_{B,w}\right \vert \left \vert f\left(  y\right)  \right \vert dy\frac
{dt}{t^{n+1}}\text{ holds.}%
\end{align*}
Applying the H\"{o}lder's inequality and by Lemma \ref{lemma0}, (\ref{10}) and
(\ref{3}), we get%
\begin{align*}
J_{1}  & \lesssim \Vert b\Vert_{\ast}w\left(  B\left(  x_{0},r\right)  \right)
^{\frac{1}{p}}\int \limits_{2r}^{\infty}\left(  1+\ln \frac{t}{r}\right)
\left \Vert \Omega \left(  \cdot-y\right)  \right \Vert _{L_{s}\left(  B\left(
x_{0},t\right)  \right)  }\left \Vert f\right \Vert _{L_{s^{\prime}}\left(
B\left(  x_{0},t\right)  \right)  }\frac{dt}{t^{n+1}}\\
& \lesssim \Vert b\Vert_{\ast}w\left(  B\left(  x_{0},r\right)  \right)
^{\frac{1}{p}}\int \limits_{2r}^{\infty}\left(  1+\ln \frac{t}{r}\right)
\left \Vert f\right \Vert _{L_{p,w}\left(  B\left(  x_{0},t\right)  \right)
}\left \Vert w^{-\frac{s^{\prime}}{p}}\right \Vert _{L_{\left(  \frac
{p}{s^{\prime}}\right)  ^{\prime}}\left(  B\left(  x_{0},t\right)  \right)
}^{\frac{1}{s^{\prime}}}\left \vert B\left(  x_{0},2t\right)  \right \vert
^{\frac{1}{s}}\frac{dt}{t^{n+1}}\\
& \lesssim \Vert b\Vert_{\ast}w\left(  B\left(  x_{0},r\right)  \right)
^{\frac{1}{p}}\int \limits_{2r}^{\infty}\left(  1+\ln \frac{t}{r}\right)
\left \Vert f\right \Vert _{L_{p,w}\left(  B\left(  x_{0},t\right)  \right)
}w\left(  B\left(  x_{0},t\right)  \right)  ^{-\frac{1}{p}}\frac{dt}{t}.
\end{align*}

In order to estimate $J_{2}$ note that%
\[
J_{2}=\left \Vert \left(  b\left(  \cdot \right)  -b_{B\left(  x_{0},t\right)
,w}\right)  \right \Vert _{L_{p,w}\left(  B\left(  x_{0},t\right)  \right)
}\int \limits_{\left(  2B\right)  ^{C}}\frac{\left \vert \Omega \left(
x-y\right)  \right \vert }{\left \vert x_{0}-y\right \vert ^{n}}\left \vert
f\left(  y\right)  \right \vert dy.
\]

By (\ref{11}) and Lemma \ref{lemma0}, we get%
\begin{align*}
J_{2}  & \lesssim \Vert b\Vert_{\ast}w\left(  B\left(  x_{0},r\right)  \right)
^{\frac{1}{p}}\int \limits_{\left(  2B\right)  ^{C}}\frac{\left \vert
\Omega \left(  x-y\right)  \right \vert }{\left \vert x_{0}-y\right \vert ^{n}%
}\left \vert f\left(  y\right)  \right \vert dy\\
& \lesssim \Vert b\Vert_{\ast}w\left(  B\left(  x_{0},r\right)  \right)
^{\frac{1}{p}}\int \limits_{2r}^{\infty}\Vert f\Vert_{L_{p,w}(B(x_{0}%
,t))}w\left(  B\left(  x_{0},t\right)  \right)  ^{-\frac{1}{p}}\frac{dt}{t}.
\end{align*}

Summing up $J_{1}$ and $J_{2}$, for all $p\in \left(  1,\infty \right)  $ we get%
\[
\left \Vert T_{\Omega,b}f_{2}\right \Vert _{L_{p,w}\left(  B\right)  }%
\lesssim \Vert b\Vert_{\ast}\,w\left(  B\left(  x_{0},r\right)  \right)
^{\frac{1}{p}}\int \limits_{2r}^{\infty}\left(  1+\ln \frac{t}{r}\right)  \Vert
f\Vert_{L_{p,w}(B(x_{0},t))}w\left(  B\left(  x_{0},t\right)  \right)
^{-\frac{1}{p}}\frac{dt}{t}%
\]

Finally, we have the following%
\[
\left \Vert T_{\Omega,b}f\right \Vert _{L_{p,w}\left(  B\right)  }%
\lesssim \left \Vert b\right \Vert _{\ast}\left \Vert f\right \Vert _{L_{p,w}%
\left(  2B\right)  }+\Vert b\Vert_{\ast}\,w\left(  B\left(  x_{0},r\right)
\right)  ^{\frac{1}{p}}\int \limits_{2r}^{\infty}\left(  1+\ln \frac{t}%
{r}\right)  \Vert f\Vert_{L_{p,w}(B(x_{0},t))}w\left(  B\left(  x_{0}%
,t\right)  \right)  ^{-\frac{1}{p}}\frac{dt}{t}.
\]

On the other hand by (\ref{12}), we have
\[
\left \Vert T_{\Omega,b}f\right \Vert _{L_{p,w}\left(  B\right)  }\lesssim \Vert
b\Vert_{\ast}\,w\left(  B\left(  x_{0},r\right)  \right)  ^{\frac{1}{p}}%
\int \limits_{2r}^{\infty}\left(  1+\ln \frac{t}{r}\right)  \Vert f\Vert
_{L_{p,w}(B(x_{0},t))}w\left(  B\left(  x_{0},t\right)  \right)  ^{-\frac
{1}{p}}\frac{dt}{t}.
\]

For the case of $1<p<s$, $w^{1-p^{\prime}}\in$ $A_{\frac{p^{\prime}}%
{s^{\prime}}}$, we can also use the same method, so we omit the details. This
completes the proof of Lemma \ref{Lemma 5}.
\end{proof}

Now we can give the following theorem (our main result).

\begin{theorem}
\label{teo15}$\left(  \text{Our main result}\right)  $ Suppose that $\Omega \in
L_{s}(S^{n-1})$, $s>1$, is homogeneous of degree zero and $T_{\Omega,b}$ is a
sublinear operator satisfying condition (\ref{e2}), bounded on $L_{p}(w) $.
Let $1<p<\infty$ and $b\in BMO\left(
%TCIMACRO{\U{211d} }%
%BeginExpansion
\mathbb{R}
%EndExpansion
^{n}\right)  $.\newline Let also, for $s^{\prime}\leq p$ and $w\in A_{\frac
{p}{s^{\prime}}}$ the pair $(\varphi_{1},\varphi_{2})$ satisfies the condition%
\begin{equation}
\int \limits_{r}^{\infty}\left(  1+\ln \frac{t}{r}\right)  \frac
{\operatorname*{essinf}\limits_{t<\tau<\infty}\varphi_{1}(x,\tau)w\left(
B\left(  x,\tau \right)  \right)  ^{\frac{1}{p}}}{w\left(  B\left(  x,t\right)
\right)  ^{\frac{1}{p}}}\frac{dt}{t}\leq C\, \varphi_{2}(x,r),\label{47}%
\end{equation}
and for $p<s$ and $w^{1-p^{\prime}}\in$ $A_{\frac{p^{\prime}}{s^{\prime}}}$
the pair $(\varphi_{1},\varphi_{2})$ satisfies the condition%
\begin{equation}
\int \limits_{r}^{\infty}\left(  1+\ln \frac{t}{r}\right)  \frac
{\operatorname*{essinf}\limits_{t<\tau<\infty}\varphi_{1}(x,\tau)\left \Vert
w\right \Vert _{L_{\frac{s}{s-p}}\left(  B\left(  x,r\right)  \right)  }%
^{\frac{1}{p}}}{\left \Vert w\right \Vert _{L_{\frac{s}{s-p}}\left(  B\left(
x,t\right)  \right)  }^{\frac{1}{p}}}\frac{dt}{t}\leq C\, \varphi
_{2}(x,r)\frac{w\left(  B\left(  x,r\right)  \right)  ^{\frac{1}{p}}%
}{\left \Vert w\right \Vert _{L_{\frac{s}{s-p}}\left(  B\left(  x,r\right)
\right)  }^{\frac{1}{p}}},\label{48}%
\end{equation}
where $C$ does not depend on $x$ and $r$.\newline Then, the operator
$T_{\Omega,b}$ is bounded from $M_{p,\varphi_{1}}\left(  w\right)  $ to
$M_{p,\varphi_{2}}\left(  w\right)  $. Moreover%
\begin{equation}
\left \Vert T_{\Omega,b}f\right \Vert _{M_{p,\varphi_{2}}\left(  w\right)
}\lesssim \left \Vert b\right \Vert _{\ast}\left \Vert f\right \Vert _{M_{p,\varphi
_{1}}\left(  w\right)  }.\label{49}%
\end{equation}

\end{theorem}

\begin{proof}
since $f\in M_{p,\varphi_{1}}\left(  w\right)  $, by (\ref{4}) and the
non-decreasing, with respect to $t$, of the norm $\left \Vert f\right \Vert
_{L_{p,w}\left(  B\left(  x_{0},t\right)  \right)  }$, we get%
\begin{align*}
& \frac{\left \Vert f\right \Vert _{L_{p,w}\left(  B\left(  x_{0},t\right)
\right)  }}{\operatorname*{essinf}\limits_{0<t<\tau<\infty}\varphi_{1}%
(x_{0},\tau)w\left(  B\left(  x_{0},\tau \right)  \right)  ^{\frac{1}{p}}}\\
& \leq \operatorname*{esssup}\limits_{0<t<\tau<\infty}\frac{\left \Vert
f\right \Vert _{L_{p,w}\left(  B\left(  x_{0},t\right)  \right)  }}{\varphi
_{1}(x_{0},\tau)w\left(  B\left(  x_{0},\tau \right)  \right)  ^{\frac{1}{p}}%
}\\
& \leq \operatorname*{esssup}\limits_{0<\tau<\infty}\frac{\left \Vert
f\right \Vert _{L_{p,w}\left(  B\left(  x_{0},\tau \right)  \right)  }}%
{\varphi_{1}(x_{0},\tau)w\left(  B\left(  x_{0},\tau \right)  \right)
^{\frac{1}{p}}}\\
& \leq \left \Vert f\right \Vert _{M_{p,\varphi_{1}}\left(  w\right)  }.
\end{align*}
For $s^{\prime}\leq p<\infty$, since $(\varphi_{1},\varphi_{2})$ satisfies
(\ref{47}), we have%
\begin{align*}
& \int \limits_{r}^{\infty}\left(  1+\ln \frac{t}{r}\right)  \left \Vert
f\right \Vert _{L_{p,w}\left(  B\left(  x_{0},t\right)  \right)  }w\left(
B\left(  x_{0},t\right)  \right)  ^{-\frac{1}{p}}\frac{dt}{t}\\
& \leq \int \limits_{r}^{\infty}\left(  1+\ln \frac{t}{r}\right)  \frac
{\left \Vert f\right \Vert _{L_{p,w}\left(  B\left(  x_{0},t\right)  \right)  }%
}{\operatorname*{essinf}\limits_{t<\tau<\infty}\varphi_{1}(x_{0},\tau)w\left(
B\left(  x_{0},\tau \right)  \right)  ^{\frac{1}{p}}}\frac
{\operatorname*{essinf}\limits_{t<\tau<\infty}\varphi_{1}(x_{0},\tau)w\left(
B\left(  x_{0},\tau \right)  \right)  ^{\frac{1}{p}}}{w\left(  B\left(
x_{0},t\right)  \right)  ^{\frac{1}{p}}}\frac{dt}{t}\\
& \leq C\left \Vert f\right \Vert _{M_{p,\varphi_{1}}\left(  w\right)  }%
\int \limits_{r}^{\infty}\left(  1+\ln \frac{t}{r}\right)  \frac
{\operatorname*{essinf}\limits_{t<\tau<\infty}\varphi_{1}(x_{0},\tau)w\left(
B\left(  x_{0},\tau \right)  \right)  ^{\frac{1}{p}}}{w\left(  B\left(
x_{0},t\right)  \right)  ^{\frac{1}{p}}}\frac{dt}{t}\\
& \leq C\left \Vert f\right \Vert _{M_{p,\varphi_{1}}\left(  w\right)  }%
\varphi_{2}(x_{0},r).
\end{align*}
Then by (\ref{40*}), we get%
\begin{align*}
\left \Vert T_{\Omega,b}f\right \Vert _{M_{p,\varphi_{2}}\left(  w\right)  }  &
=\sup_{x_{0}\in{\mathbb{R}^{n},}r>0}\varphi_{2}\left(  x_{0},r\right)
^{-1}w\left(  B\left(  x_{0},r\right)  \right)  ^{-\frac{1}{p}}\left \Vert
T_{\Omega,b}f\right \Vert _{L_{p,w}\left(  B\left(  x_{0},r\right)  \right)
}\\
& \leq C\left \Vert b\right \Vert _{\ast}\sup_{x_{0}\in{\mathbb{R}^{n},}%
r>0}\varphi_{2}\left(  x_{0},r\right)  ^{-1}\int \limits_{r}^{\infty}\left(
1+\ln \frac{t}{r}\right)  \left \Vert f\right \Vert _{L_{p,w}\left(  B\left(
x_{0},t\right)  \right)  }w\left(  B\left(  x_{0},t\right)  \right)
^{-\frac{1}{p}}\frac{dt}{t}\\
& \leq C\left \Vert b\right \Vert _{\ast}\left \Vert f\right \Vert _{M_{p,\varphi
_{1}}\left(  w\right)  }.
\end{align*}
For the case of $1<p<s$, we can also use the same method, so we omit the
details. This completes the proof of Theorem \ref{teo15}.
\end{proof}

For the sublinear commutator of the fractional maximal operator with rough
kernel which is defined as follows%

\[
M_{\Omega,b}\left(  f\right)  (x)=\sup_{t>0}|B(x,t)|^{-1}\int \limits_{B(x,t)}%
\left \vert b\left(  x\right)  -b\left(  y\right)  \right \vert \left \vert
\Omega \left(  x-y\right)  \right \vert |f(y)|dy
\]
and for the linear commutator of the singular integral $[b,\overline
{T}_{\Omega}]$ by Theorem \ref{teo15}, we get the following new result.

\begin{corollary}
\label{corollary 12}Suppose that $\Omega \in L_{s}(S^{n-1})$, $s>1$, is
homogeneous of degree zero, $1<p<\infty$ and $b\in BMO\left(
%TCIMACRO{\U{211d} }%
%BeginExpansion
\mathbb{R}
%EndExpansion
^{n}\right)  $. If for $s^{\prime}\leq p$ and $w\in A_{\frac{p}{s^{\prime}}} $
the pair $(\varphi_{1},\varphi_{2})$ satisfies condition (\ref{47}) and for
$p<s$ and $w^{1-p^{\prime}}\in$ $A_{\frac{p^{\prime}}{s^{\prime}}}$ the pair
$(\varphi_{1},\varphi_{2})$ satisfies condition (\ref{48}). Then, the
operators $M_{\Omega,b}$ and $[b,\overline{T}_{\Omega}]$ are bounded from
$M_{p,\varphi_{1}}\left(  w\right)  $ to $M_{p,\varphi_{2}}\left(  w\right)  $.
\end{corollary}

In the case of $w=1$ from Theorem \ref{teo15}, we get

\begin{corollary}
$\left(  \text{see \cite{BGGS, Gurbuz}}\right)  $ Suppose that $\Omega \in
L_{s}(S^{n-1})$, $s>1$, is homogeneous of degree zero and $T_{\Omega,b}$ is a
sublinear operator satisfying condition (\ref{e2}), bounded on $L_{p}%
({\mathbb{R}^{n}})$. Let $1<p<\infty$ and $b\in BMO\left(
%TCIMACRO{\U{211d} }%
%BeginExpansion
\mathbb{R}
%EndExpansion
^{n}\right)  $.\newline Let also, for $s^{\prime}\leq p$ the pair
$(\varphi_{1},\varphi_{2})$ satisfies the condition%
\[
\int \limits_{r}^{\infty}\left(  1+\ln \frac{t}{r}\right)  \frac
{\operatorname*{essinf}\limits_{t<\tau<\infty}\varphi_{1}\left(
x,\tau \right)  \tau^{\frac{n}{p}}}{t^{\frac{n}{p}+1}}dt\leq C\varphi
_{2}\left(  x,r\right)  ,
\]
and for $p<s$ the pair $(\varphi_{1},\varphi_{2})$ satisfies the condition%
\[
\int \limits_{r}^{\infty}\left(  1+\ln \frac{t}{r}\right)  \frac
{\operatorname*{essinf}\limits_{t<\tau<\infty}\varphi_{1}\left(
x,\tau \right)  \tau^{\frac{n}{p}}}{t^{\frac{n}{p}-\frac{n}{s}+1}}dt\leq
C\varphi_{2}\left(  x,r\right)  r^{\frac{n}{s}},
\]
where $C$ does not depend on $x$ and $r$.\newline Then, the operator
$T_{\Omega,b}$ is bounded from $M_{p,\varphi_{1}}$ to $M_{p,\varphi_{2}}$.
Moreover%
\[
\left \Vert T_{\Omega,b}f\right \Vert _{M_{p,\varphi_{2}}}\lesssim \left \Vert
b\right \Vert _{\ast}\left \Vert f\right \Vert _{M_{p,\varphi_{1}}}.
\]

\end{corollary}

In the case of $\varphi_{1}\left(  x,r\right)  =\varphi_{2}\left(  x,r\right)
\equiv w\left(  B\left(  x,r\right)  \right)  ^{^{\frac{\kappa-1}{p}}}$ from
Theorem \ref{teo15}, we get the following new result.

\begin{corollary}
Let $1<p<\infty$, $\Omega \in L_{s}\left(  S^{n-1}\right)  $, $s>1$, be
homogeneous of degree zero, $0<\kappa<1$ and $b\in BMO\left(
%TCIMACRO{\U{211d} }%
%BeginExpansion
\mathbb{R}
%EndExpansion
^{n}\right)  $. Let also $T_{\Omega,b}$ be a sublinear operator satisfying
condition (\ref{e2}) and bounded on $L_{p}(w)$. For $s^{\prime}\leq p$ and
$w\in A_{\frac{p}{s^{\prime}}}$ or $p<s$ and $w^{1-p^{\prime}}\in
A_{\frac{p^{\prime}}{s^{\prime}}}$, the operator $T_{\Omega,b}$ is bounded on
the weighted Morrey spaces $L_{p,\kappa}(w)$.
\end{corollary}

When $\Omega \equiv1$, from Theorem \ref{teo15} we get

\begin{corollary}
\label{corollary 5*}Let $1<p<\infty$, $w\in A_{p}$, $b\in BMO\left(
%TCIMACRO{\U{211d} }%
%BeginExpansion
\mathbb{R}
%EndExpansion
^{n}\right)  $ and the pair $(\varphi_{1},\varphi_{2})$ satisfies condition
(\ref{47}). Let also $T_{b}$ be a sublinear operator satisfying condition
(\ref{e2}) and bounded on $L_{p}(w)$. Then the operator $T_{b}$ is bounded
from $M_{p,\varphi_{1}}\left(  w\right)  $ to $M_{p,\varphi_{2}}\left(
w\right)  $.
\end{corollary}

\begin{remark}
Corollary \ref{corollary 5*} has been proved in \cite{Karaman}.
\end{remark}

When $\Omega \equiv1$, in the case of $\varphi_{1}\left(  x,r\right)
=\varphi_{2}\left(  x,r\right)  \equiv w\left(  B\left(  x,r\right)  \right)
^{^{\frac{\kappa-1}{p}}}$, from Theorem \ref{teo15} we get the following new result.

\begin{corollary}
\label{corollary 4*}$1<p<\infty$, $0<\kappa<1$, $w\in A_{p}$ and $b\in
BMO\left(
%TCIMACRO{\U{211d} }%
%BeginExpansion
\mathbb{R}
%EndExpansion
^{n}\right)  $. Let also $T_{b}$ be a sublinear operator satisfying condition
(\ref{e2}) and bounded on $L_{p}(w)$. Then the operator $T_{b}$ is bounded on
the weighted Morrey spaces $L_{p,\kappa}(w)$.
\end{corollary}

\begin{remark}
Note that, from Corollary \ref{corollary 4*} for the operators $M_{b}$ and
$[b,\overline{T}]$ we get results which are proved in \cite{KomShir}.
\end{remark}

\begin{conclusion}
Let $1<p<\infty$, $0<\kappa<1$, $w\in A_{p}$ and $b\in BMO\left(
%TCIMACRO{\U{211d} }%
%BeginExpansion
\mathbb{R}
%EndExpansion
^{n}\right)  $. Then, the operators $M_{b}$ and $[b,\overline{T}]$ are bounded
on the weighted Morrey spaces $L_{p,\kappa}(w)$.
\end{conclusion}

Now, we give the applications of Theorem \ref{teo9} and Theorem \ref{teo15}
for the Marcinkiewicz operator .

Let $S^{n-1}=\{x\in{\mathbb{R}^{n}}:|x|=1\}$ be the unit sphere in
${\mathbb{R}^{n}}$ equipped with the Lebesgue measure $d\sigma$. Suppose that
$\Omega$ satisfies the following conditions.

(a) $\Omega$ is the homogeneous function of degree zero on ${\mathbb{R}^{n}%
}\setminus \{0\}$, that is,
\[
\Omega(\mu x)=\Omega(x),~~\text{for any}~~ \mu>0,x\in{\mathbb{R}^{n}}%
\setminus \{0\}.
\]

(b) $\Omega$ has mean zero on $S^{n-1}$, that is,
\[
\int \limits_{S^{n-1}}\Omega(x^{\prime})d\sigma(x^{\prime})=0,
\]
where $x^{\prime}=\frac{x}{\left \vert x\right \vert }$ for any $x\neq0$.

(c) $\Omega \in Lip_{\gamma}(S^{n-1})$, $0<\gamma \leq1$, that is there exists a
constant $M>0$ such that,
\[
|\Omega(x^{\prime})-\Omega(y^{\prime})|\leq M|x^{\prime}-y^{\prime}|^{\gamma
}~~\text{for any}~~x^{\prime},y^{\prime}\in S^{n-1}.
\]

In 1958, Stein \cite{Stein58} defined the Marcinkiewicz integral of higher
dimension $\mu_{\Omega}$ as
\[
\mu_{\Omega}(f)(x)=\left(  \int \limits_{0}^{\infty}|F_{\Omega,t}%
(f)(x)|^{2}\frac{dt}{t^{3}}\right)  ^{1/2},
\]
where
\[
F_{\Omega,t}(f)(x)=\int \limits_{|x-y|\leq t}\frac{\Omega(x-y)}{|x-y|^{n-1}%
}f(y)dy.
\]

Since Stein's work in 1958, the continuity of Marcinkiewicz integral has been
extensively studied as a research topic and also provides useful tools in
harmonic analysis \cite{LuDingY, St, Stein93, Torch}.

The sublinear commutator of the operator $\mu_{\Omega}$ is defined by
\[
\lbrack b,\mu_{\Omega}](f)(x)=\left(
%TCIMACRO{\dint \limits_{0}^{\infty}}%
%BeginExpansion
{\displaystyle \int \limits_{0}^{\infty}}
%EndExpansion
|F_{\Omega,t,b}(f)(x)|^{2}\frac{dt}{t^{3}}\right)  ^{1/2},
\]
where
\[
F_{\Omega,t,b}(f)(x)=%
%TCIMACRO{\dint \limits_{|x-y|\leq t}}%
%BeginExpansion
{\displaystyle \int \limits_{|x-y|\leq t}}
%EndExpansion
\frac{\Omega(x-y)}{|x-y|^{n-1}}[b(x)-b(y)]f(y)dy.
\]

We consider the space $H=\{h:\Vert h\Vert=\left(  \int \limits_{0}^{\infty
}|h(t)|^{2}\frac{dt}{t^{3}}\right)  ^{1/2}<\infty \}$. Then, it is clear that
$\mu_{\Omega}(f)(x)=\Vert F_{\Omega,t}(x)\Vert$.

By the Minkowski inequality and the conditions on $\Omega$, we get
\[
\mu_{\Omega}(f)(x)\leq \int \limits_{{\mathbb{R}^{n}}}\frac{|\Omega
(x-y)|}{|x-y|^{n-1}}|f(y)|\left(  \int \limits_{|x-y|}^{\infty}\frac{dt}{t^{3}%
}\right)  ^{1/2}dy\leq C\int \limits_{{\mathbb{R}^{n}}}\frac{\left \vert
\Omega(x-y)\right \vert }{|x-y|^{n}}|f(y)|dy.
\]
Thus, $\mu_{\Omega}$ satisfies the condition (\ref{e1}). It is known that
$\mu_{\Omega}$ and $[b,\mu_{\Omega}]$ are bounded on $L_{p}(w)$ for
$s^{\prime}\leq p$ and $w\in A_{\frac{p}{s^{\prime}}}$ or for $1<p<s$ and
$w^{1-p^{\prime}}\in$ $A_{\frac{p^{\prime}}{s^{\prime}}}$ (see
\cite{Ding-Fan-Pan, Ding-Lu-Yabuta}), then from Theorems \ref{teo9} and
\ref{teo15} we get

\begin{corollary}
Let $\Omega \in L_{s}(S^{n-1})$, $s>1$, $1\leq p<\infty$. Let also, for
$s^{\prime}\leq p$ and $w\in A_{\frac{p}{s^{\prime}}}$ the pair $(\varphi
_{1},\varphi_{2})$ satisfies condition (\ref{316}) and for $p<s$ and
$w^{1-p^{\prime}}\in$ $A_{\frac{p^{\prime}}{s^{\prime}}}$ the pair
$(\varphi_{1},\varphi_{2})$ satisfies condition (\ref{317}) and $\Omega$
satisfies conditions (a)--(c). Then the operator $\mu_{\Omega}$ is bounded
from $M_{p,\varphi_{1}}\left(  w\right)  $ to $M_{p,\varphi_{2}}\left(
w\right)  $ for $p>1$ and from $M_{1,\varphi_{1}}\left(  w\right)  $to
$WM_{1,\varphi_{2}}\left(  w\right)  $ for $p=1$.
\end{corollary}

\begin{corollary}
Let $\Omega \in L_{s}(S^{n-1})$, $s>1$, $1\leq p<\infty$, $0<\kappa<1$ and
$w\in A_{\frac{p}{s^{\prime}}}$ or $w^{1-p^{\prime}}\in$ $A_{\frac{p^{\prime}%
}{s^{\prime}}}$. Suppose that $\Omega$ satisfies conditions (a)--(c). Then the
operator $\mu_{\Omega}$ is bounded on $L_{p,\kappa}(w)$ for $p>1$ and bounded
from $L_{1,\kappa}(w)$ to $WL_{1,\kappa}(w)$.
\end{corollary}

\begin{corollary}
Let $\Omega \in L_{s}(S^{n-1})$, $s>1$. Let $1<p<\infty$ and $b\in BMO\left(
%TCIMACRO{\U{211d} }%
%BeginExpansion
\mathbb{R}
%EndExpansion
^{n}\right)  $. Let also, for $s^{\prime}\leq p$ and $w\in A_{\frac
{p}{s^{\prime}}}$ the pair $(\varphi_{1},\varphi_{2})$ satisfies condition
(\ref{47}) and for $p<s$ and $w^{1-p^{\prime}}\in$ $A_{\frac{p^{\prime}%
}{s^{\prime}}}$ the pair $(\varphi_{1},\varphi_{2})$ satisfies condition
(\ref{48}) and $\Omega$ satisfies conditions (a)--(c). Then, the operator
$[b,\mu_{\Omega}]$ is bounded from $M_{p,\varphi_{1}}\left(  w\right)  $ to
$M_{p,\varphi_{2}}\left(  w\right)  $.
\end{corollary}

\begin{corollary}
Let $\Omega \in L_{s}(S^{n-1})$, $s>1$, $1<p<\infty$, $0<\kappa<1$, $w\in
A_{\frac{p}{s^{\prime}}}$ or $w^{1-p^{\prime}}\in$ $A_{\frac{p^{\prime}%
}{s^{\prime}}}$ and $b\in BMO\left(
%TCIMACRO{\U{211d} }%
%BeginExpansion
\mathbb{R}
%EndExpansion
^{n}\right)  $. Suppose that $\Omega$ satisfies conditions (a)--(c). Then the
operator $[b,\mu_{\Omega}]$ is bounded on $L_{p,\kappa}(w)$.
\end{corollary}

\end{document}